\documentclass[12pt]{amsart}

\usepackage{amsmath}
\usepackage{amssymb}
\usepackage{graphicx}
\usepackage{mathtools}
\usepackage{appendix}
\usepackage{hyperref}
\usepackage{multirow}
\usepackage{caption} 
\newtheorem{theorem}{Theorem}[section]
\newtheorem{lemma}[theorem]{Lemma}

\newtheorem{corollary}[theorem]{Corollary}

\theoremstyle{definition}
\newtheorem{definition}[theorem]{Definition}
\newtheorem{example}[theorem]{Example}
\newtheorem{remark}[theorem]{Remark}

\renewcommand{\Re}{\operatorname{Re}} 
\renewcommand{\Im}{\operatorname{Im}}

\DeclareMathOperator{\diam}{diam}

\hypersetup{
    colorlinks=true,
    allcolors=blue,
    pdftitle={Overleaf Example},
    pdfpagemode=FullScreen,
}
 
\title{Intersections of Cantor sets derived from Complex Radix Expansions} 
\author{Neil MacVicar}

\begin{document} 
  
\begin{abstract} 
Let $C$ be the  attractor of the IFS $\{f_{d}(z) = (-n+i)^{-1}(z+d): d\in D\}$, $D\subset\{0, 1, \ldots, n^{2}\}$ and let $\dim$ denote the box-counting dimension. It is known that for all $\lambda\in[0, 1]$, that the set of complex numbers $\alpha$ for which $\dim(C\cap(C+\alpha)) = \lambda\dim(C)$ is dense in the set of $\alpha$ for which $C \cap (C + \alpha) \neq \emptyset$ when $d \leq n^{2}/2$ for all $d\in D$ and $|\delta - \delta^{'}| > n$ for all $\delta \neq \delta^{'} \in D - D$. We show that this result still holds when we replace $|\delta - \delta^{'}| > n$ with $|\delta - \delta^{'}| > 1$. In fact, for sufficiently large $n$, the result even holds when we remove the assumption $d\leq n^{2}/2$ and replace $|\delta - \delta^{'}| > n$ by $|\delta - \delta^{'}| > 2$. Additionally, we make similar statements where $\dim$ denotes the Hausdorff dimension or packing dimension. Our insights also find application in classifying the self-similarity of $C\cap(C+\alpha)$. Namely we connect the occurrence of self-similarity to the notion of strongly eventually periodic sequences seen for analogous objects on the real line. We also provide a new proof of a result of W. Gilbert that inspired this work. 
\end{abstract}

\maketitle


\section{\textbf{Introduction}} \label{section:into} 

Given a metric space $(X, d)$, a function $f:X \rightarrow X$ is a \textit{contraction} if there exists a real number $c\in(0, 1)$ such that 
\begin{equation}\label{eq:contract}
d(f(x), f(y)) \leq cd(x, y)
\end{equation}
for all $x, y\in X$. We call $c$ the \textit{contraction factor}. Given a finite collection of contractions $\{f_{1}, f_{2}, \ldots, f_{N}\}$ on a complete metric space $X$, called an \textit{iterated function system} (IFS), there exists a unique compact set $A \subset X$ that satisfies $\cup_{i=1}^{N}f_{i}(A) = A$ \cite{H81}. The set $A$ is called the \textit{attractor} of the IFS. A contraction is called a \textit{similarity} if we have equality in (\ref{eq:contract}). The attractor of an IFS is called \textit{self-similar} when all the contractions are similarities. Many examples of self-similar sets are fractals. That is, they exhibit non-integer Hausdorff dimension. For example, the middle third Cantor set $C_{3}$ is the attractor of the IFS $\{x \mapsto x/3, x \mapsto x/3+2/3\}$ and has Hausdorff dimension $\log{2}/\log{3}$. 

Consider the following property of the middle third Cantor set proved by Davis and Hu in \cite{DH95}. Let $\dim_{H}X$ denote Hausdorff dimension of a set $X\subset\mathbb{R}^{m}$. See Definition~\ref{def:hausDim} to recall how Hausdorff dimension is formulated. Let $F$ be the set of real numbers $\alpha$ for which $C_{3}\cap(C_{3}+\alpha)$ is nonempty.  The level sets of the function from $F$ to $[0, \log{2}/\log{3}]$ given by $\alpha \mapsto \dim_{H}(C_{3}\cap(C_{3}+\alpha))$ is dense in $F$. Pedersen and Shaw proved a result similar to that of Davis and Hu for the attractor of a set of similarities defined on the complex plane using the box-counting dimension instead of the Hausdorff dimension \cite{PS21}. To describe the IFS, fix a positive integer $n\geq2$, let $b := -n+i$ and let $D\subset\{0, 1, \ldots, n^{2}\}$. The IFS studied by Pedersen and Shaw in \cite{PS21} is the collection of functions $f_{d}:\mathbb{C}\rightarrow\mathbb{C}$ given by $f_{d}(z) = b^{-1}(z + d)$ where $d\in D$.

The choices of factor $-n+i$ and digits $\{0, 1, \ldots, n^{2}\}$ are special. Its primary feature is captured by the following theorem of Katai and Szabo, \cite{KS74}). 

\begin{theorem}[I. Katai, J. Szabo, \cite{KS74}, theorem 1] 
Given a Gaussian integer $b$, every Gaussian integer $g$ can be uniquely written as
\begin{equation}\label{eq:integerRep} 
g = \lambda_{0} + \lambda_{1}b + \ldots + \lambda_{k}b^{k}. 
\end{equation}
with $\lambda_{j} \in \{0, 1, 2, \ldots, |b|^{2}-1\}$ if and only if $\Re(b) < 0$ and $\Im(b) = \pm 1$. 
\end{theorem}

In other words, the base must be of the form $b = -n\pm i$ where $n$ is a positive integer and the set of digits is $\{0, 1, \ldots, n^{2}\}$. This result can be used to prove a result about complex radix expansions, found in \cite{KS74}. 

\begin{definition} 
Let $b$ be a Gaussian integer and $D \subset \mathbb{Z}$. We call any infinite series of the form
\begin{equation}
z = d_{\ell}b^{\ell} + d_{\ell-1}b^{\ell-1} + \cdots + d_{0} + \sum_{j=1}^{\infty}d_{-j}b^{-j}.
\end{equation}
where $\ell$ is some integer and $d_{k}\in D$ for all $k \leq \ell$ a \textit{radix expansion in base $(b, D)$}. 
\end{definition}

\begin{theorem}[I. Katai, J. Szabo, \cite{KS74}, theorem 2] \label{thm:radixExist} 
Suppose $n$ is a positive integer and set $b = -n+i$. Every complex number has a radix expansion in base $(b, \{0, 1, \ldots, n^{2}\})$. 
\end{theorem}

The radix expansions in base $(-n+i, \{0, 1, \ldots, n^{2}\})$ featured in Theorem~\ref{thm:radixExist} are not always unique. In fact, complex numbers can have up to three distinct radix expansions in base $(-n+i, \{0, 1, \ldots, n^{2}\})$ \cite{G82}. The rules for when two distinct radix expansions in base $(-n+i, \{0, 1, \ldots, n^{2}\})$ evaluate to the same complex number are significantly more complicated and present challenges that are not present when working with real numbers. Pedersen and Shaw grappled with some of these challenges in \cite{PS21}. 

To state the theorem proved by Pedersen and Shaw in \cite{PS21}, we define a function analogous to the function studied by Davis and Hu in \cite{DH95}. We provide several supporting definitions that are also used in later sections of this paper. 

\begin{definition}\label{def:restrictedDigitSet} 
Fix an integer $n\geq1$, let $b := -n +i$ and suppose $D\subset\{0, 1, \ldots, n^{2}\}$. We call the attractor of the IFS given by $\{f_{d}(x) = b^{-1}(x+d) : d\in D\}$, the \textit{restricted digit set generated by $(n, D)$} and denote it by $C_{n, D}$. In the special case that $D = \{0, 1, \ldots, n^{2}\}$, then we call $C_{n, D}$ the \textit{$n$th fundamental tile} and denote it by $T_{n}$.
\end{definition}

\begin{definition}\label{def:fundTrans} 
Fix an integer $n\geq2$ and suppose $D\subset\{0, 1, \ldots, n^{2}\}$. We call the set of $\alpha\in\mathbb{C}$ such that $C_{n, D}\cap (C_{n, D} + \alpha)$ is nonempty the \textit{fundamental set of translations generated by $(n, D)$} and denote it by $F_{n, D}$. 
\end{definition}

Let us recall the definition of the box-counting dimension for subsets of Euclidean space. 

\begin{definition}\label{def:boxcounting} 
Let $X$ be a bounded subset of $\mathbb{R}^{m}$. Given $\delta>0$, we let $N_{\delta}(X)$ denote the smallest number of sets of diameter $\delta$
needed to cover $X$. The \textit{upper box-counting dimension} and the \textit{lower box-counting dimension of $X$} are
\begin{align}
\overline{\dim}_{B}X &:= \limsup_{\delta\rightarrow0}\frac{\log N_{\delta}(X)}{-\log\delta}, \\
\underline{\dim}_{B}X &:= \liminf_{\delta\rightarrow0}\frac{\log N_{\delta}(X)}{-\log\delta}, 
\end{align}
respectively. If these quantities are equal, then that value is \textit{the box-counting dimension of $X$} and is denoted by $\dim_{B}X$. 
\end{definition}

The box-counting dimension may not always exist. For any positive integer $n\geq1$ and subset $D\subset\{0, 1, \ldots, n^{2}\}$, $C_{n, D}$ is self-similar by definition and thus its box-counting dimension exists (see corollary 3.3 in \cite{F97}). 

\begin{definition}\label{def:dimFunc} 
Fix an integer $n\geq2$ and suppose $D\subset\{0, 1, \ldots, n^{2}\}$. Let $C(\alpha) := C_{n, D} \cap (C_{n, D} + \alpha)$ for $\alpha \in \mathbb{C}$. We define the function 
\begin{align}
&\Phi_{n, D} : \{\alpha\in F_{n, D}: \dim_{B}C(\alpha) \;\; \text{exists}\} \rightarrow [0, \dim_{B}C_{n, D}] \\
&\Phi_{n, D}(\alpha) = \dim_{B}C(\alpha).
\end{align}
\end{definition}

We can now state the theorem in \cite{PS21} and our own theorem. 

\begin{theorem}[S. Pedersen, V. Shaw, \cite{PS21}, corollary 7.5] \label{thm:oldBound} 
Fix an integer $n\geq3$. Suppose $D\subset\{0, 1, \ldots, \lfloor{n^{2}/2}\rfloor\}$ satisfies the condition $|\delta - \delta^{'}| > n$ for all $\delta \neq \delta^{'}$ in $D-D$. Then the level sets of $\Phi_{n, D}$ are dense in $F_{n, D}$.
\end{theorem} 

In this paper, we demonstrate that the condition $|\delta - \delta^{'}| > n$ for all $\delta \neq \delta^{'}$ in $D-D$ can be significantly relaxed. In fact, it is possible to achieve the result with a lower bound that does not depend on $n$. We state our theorem using the same notation as in Theorem~\ref{thm:oldBound}. 

\begin{theorem}\label{thm:improvement} 
Suppose one of the following holds
\begin{itemize}
\item[(i)] $n$ is a positive integer greater than or equal to $2$ and $D\subset\{0, 1, \ldots, \lfloor n^{2}/2 \rfloor\}$ satisfies $|\delta - \delta^{'}| > 1$ for all $\delta \neq \delta^{'}\in\Delta := D - D$. 
\item[(ii)] $n$ is a positive integer greater than or equal to $5$ and $D\subset\{0, 1, \ldots, n^{2}\}$ satisfies $|\delta - \delta^{'}| > 2$ for all $\delta \neq \delta^{'}\in\Delta$. 
\item[(iii)] $n$ is one of $2, 3$ or $4$ and $D\subset\{0, 1, \ldots, n^{2}\}$ satisfies $|\delta - \delta^{'}| > 3$ for all $\delta \neq \delta^{'}\in\Delta$. 
\end{itemize}
Then the level sets of $\Phi_{n, D}$ are dense in $F_{n, D}$.
\end{theorem}

 If any of (i), (ii), or (iii) hold, then the level sets of the maps obtained by replacing the box-counting dimension in Definition~\ref{def:dimFunc} by Hausdorff dimension or packing dimension are also dense in $F_{n, D}$. This is because $C(\alpha)$ is of Moran type. We discuss this in Section~\ref{sec:boxIntersect}. We provide definitions of Hausdorff dimension and packing dimension in Section~\ref{section:Background}. 

The conditions present in (ii) and (iii) are paired frequently enough that it is convenient to package them as a single definition. 

\begin{definition} 
Fix a positive integer $n\geq 2$ and let $D$ be a subset of $\{0, 1, \ldots, n^{2}\}$. The set $\Delta := D-D$ is called \textit{sparse} if for all $\delta, \delta^{'}\in\Delta$ either $|\delta-\delta^{'}| > 2$ when $n\geq5$ or $|\delta-\delta^{'}| > 3$ when $n = 2, 3,$ or $4$. 
\end{definition}

Recall that for a fixed integer $n\geq 2$ and $D\subset\{0, 1, \ldots, n^{2}\}$, we use $C_{n, D}$ to denote the attractor of the IFS containing all the maps $f_{d}(x) = b^{-1}(x+d)$ where $d\in D$ and $b := -n+i$. The conditions on the distance between elements of $\Delta$ controls the overlaps among the images $f_{d}(C_{n, D}), d\in D$. If $f_{d}(C_{n, D}) \cap f_{d^{'}}(C_{n, D})$ is empty for all $d\neq d^{'}$, then the IFS defined by $D$ is said to satisfy the \textit{strong separation condition} (SSC). Once a lower bound on the differences $|\delta-\delta^{'}|$ is found to imply the SSC, we can employ a technique present in \cite{PS21} to conclude that the level sets of $\Phi_{n, D}$ are dense. The original condition $|\delta - \delta^{'}| > n$ for all $\delta \neq \delta^{'}$ in $\Delta$ was achieved without leveraging the properties of radix expansions in base $(-n+i, \{0, 1, \ldots, n^{2}\})$. Those properties are intimately connected to the structure of $C_{n, D}$. We achieved (i) by applying the following result of Gilbert \cite{G82}. A similar statement can be made for the special case $n = 2$. 

\begin{theorem}[W. Gilbert, \cite{G82}, proposition 1]\label{thm:gil3} 
Fix an integer $n \geq 3$. Let $T$ be the attractor of the IFS $f_{d}(z) = (-n+i)^{-1}(z+d)$ where $d\in\{0, 1, \ldots, n^{2}\}$. The intersection $T\cap(T + s)$ is nonempty when $s$ is a Gaussian integer if and only if $s\in\{0, \pm 1, \pm (n + i), \pm (n-1+i)\}$. 
\end{theorem}

Theorem~\ref{thm:gil3} is used in \cite{G82} to derive a graph that describes when two distinct radix expansions in base $(-n+i, \{0, 1, \ldots, n^{2}\})$ evaluate to the same complex number. Suppose that for a fixed integer $n\geq2$, $(d_{j})_{j=1}^{\infty}$ and $(e_{j})_{j=1}^{\infty}$ are distinct sequences in $\{0, 1, \ldots, n^{2}\}^{\mathbb{N}}$ and $J$ is the first index at which the sequences differ. If we have $\sum_{j=1}^{\infty}d_{j}b^{-j} = \sum_{j=1}^{\infty}e_{j}b^{-j}$ where $b:=-n+i$, then $|d_{J} - e_{J}| = 1$ is a consequence of Theorem~\ref{thm:gil3}. Hence the separation condition of $1$ present in (i) of Theorem~\ref{thm:improvement}.We take this further by determining the upper bound on $|d_{J} - e_{J}|$ when $d_{j}, e_{j}\in \{0, \pm 1, \ldots, \pm n^{2}\})$ for all $j$. This investigation revealed the conditions (ii) and (iii) in Theorem~\ref{thm:improvement}. In producing this work, we also give a new proof of Theorem~\ref{thm:gil3} that differs from Gilbert's approach. 

The results underpinning Theorem~\ref{thm:improvement} can also be used to partially address when $C(\alpha) := C_{n, D} \cap (C_{n, D}+\alpha)$ is self-similar. This is of interest because self-similarity can somtimes be leveraged to make the computation of Hausdorff and box-counting dimension significantly easier (see Theorem~\ref{thm:strongSimDim}). This was first achieved for the middle third Cantor set in \cite{DHW08}. Our results take inspiration from the more general approaches found in \cite{LYZ11} and \cite{PP14}. The existing results connect the self-similarity of $C(\alpha)$ to a property called \textit{strong eventual periodicity}. 

\begin{definition} 
A sequence $(a_{j})_{j\geq1}$ of integers is \textit{strongly eventually periodic} (SEP) if there exists a finite sequence $(b_{\ell})_{\ell = 1}^{p}$ and a nonnegative sequence $(c_{\ell})_{\ell = 1}^{p}$, where $p$ is a positive integer, such that
\begin{equation}
(a_{j})_{j\geq1} = (b_{\ell})\overline{(b_{\ell} + c_{\ell})_{\ell = 1}^{p}}, 
\end{equation}
where $\overline{(d_{\ell})_{\ell = 1}^{p}}$ denotes the infinite repetition of the finite sequence $(d_{\ell})_{\ell = 1}^{p}$.   
\end{definition}

Suppose $C_{3}$ is the middle third Cantor set ($D = \{0, 2\}$) and that $\alpha$ is chosen such that $C_{3} \cap (C_{3} + \alpha)$ is nonempty. It must be the case that $\alpha = \sum_{j=1}^{\infty}\alpha_{j}3^{-j}$ where $\alpha_{j}\in\{-2, 0, 2\}$. It is known that $C_{3} \cap (C_{3} + \alpha)$ is self-similar if and only if $(2-|\alpha_{j}|)_{j=1}^{\infty}$ is SEP \cite{DHW08}. 

We establish analogous results for the self-similar sets $C_{n, D}$ implicit in Theorem~\ref{thm:improvement}.

\begin{definition}\label{def:ExtDigitSet} 
Fix an integer $n\geq2$, let $b := -n +i$, and suppose $D\subset\{0, 1, \ldots,  n^{2}\}$. We call the attractor of the IFS given by $\{f_{\delta}(x) = b^{-1}(x+\delta) : \delta\in D-D\}$ the \textit{extended restricted digit set generated by $(n, D)$} and denote it by $E_{n, D}$. In the special case that $D = \{0, 1, \ldots, n^{2}\}$, then we call $E_{n, D}$ the \textit{$n$th extended tile} and denote it by $\mathcal{E}_{n}$.
\end{definition}

\begin{definition} 
Let $b$ be a Gaussian integer. We call the function $\pi_{b} : \mathbb{Z}^{\mathbb{N}}\rightarrow\mathbb{C}$ given by $\pi(d_{j})_{j=1}^{\infty} = \pi_{b}(d_{j})_{j=1}^{\infty} := \sum_{j=1}^{\infty}d_{j}b^{-j}$ the $b$-\textit{coding map}. 
\end{definition}

We will drop the subscript $b$ when it is clear from context.

\begin{theorem}\label{thm:sepSp} 
Fix an integer $n\geq 2$ and let $b:= -n+i$. Suppose $D = \{0, m\}$ where $2 \leq m \leq n^{2}$ and that $\alpha\in E_{n, D}$ is chosen such that $\alpha$ has a unique radix expansion in base $(b, \{0, \pm m\})$. Let $\gamma := \pi(\gamma_{j})_{j=1}^{\infty}$, where $\gamma_{j} = \min(\{0, m\}\cap(\{0, m\}+\alpha_{j}))$ for each $j$.

If $C(\alpha) := C_{n, D} \cap (C_{n, D} + \alpha)$ is self-similar and is the attractor of an IFS containing the similarity $f(x) = rx + (1 - r)\gamma$, then $(m -|\alpha_{j}|)_{j=1}^{\infty}$ is SEP. Conversely, if $(m -|\alpha_{j}|)_{j=1}^{\infty}$ is SEP, then $C(\alpha)$ is self-similar. 
\end{theorem}

We can also discuss $C_{n, D}$ when $D$ contains more than two elements. We first introduce a version of strong eventual periodicity for sequences of sets. 

\begin{definition} 
A sequence $(A_{j})_{j=1}^{\infty}$ of nonempty subsets of the integers is called \textit{strongly eventually periodic} (SEP) if there exist two finite sequences of sets $(B_{\ell})_{\ell = 1}^{p}$ and $(C_{\ell})_{\ell = 1}^{p}$, where $p$ is a positive integer, such that 
\begin{equation}
(A_{j})_{j=1}^{\infty} = (B_{\ell})\overline{(B_{\ell} + C_{\ell})_{\ell = 1}^{p}}, 
\end{equation}
where $B + C = \{b + c : b\in B, c\in C\}$ and $\overline{(D_{\ell})_{\ell = 1}^{p}}$ denotes the infinite repetition of the finite sequence of sets $(D_{\ell})_{\ell = 1}^{p}$. 
\end{definition}

\begin{theorem}\label{thm:sepGen} 
Fix an integer $n\geq2$, let $b:=-n+i$, and suppose $D\subset\{0, 1, \ldots, n^{2}\}$ is such that $\Delta := D - D$ is sparse. Let $\alpha = \pi(\alpha_{j})_{j=1}^{\infty}$ where $\alpha_{j}\in\Delta$. The set $C(\alpha) := C_{n, D} \cap (C_{n, D} + \alpha)$ is self similar and is the attractor of an IFS of the form $\{f_{i}(x) = b^{-p}x + u_{i}\}$ where $p$ is a positive integer and each $u_{i}$ is a complex number if and only if the sequence of sets $((D\cap (D + \alpha_{j})) - \beta_{j})_{j=1}^{\infty}$ is SEP for some $\beta = \pi(\beta_{j})_{j=1}^{\infty} \in C(\alpha)$. 
\end{theorem}

We outline the remaining sections of this paper. 

\begin{enumerate}
\item In Section 2 we set our notation and recall some definitions and results. 
\item In Section 3 we derive the separation conditions present in Theorems~\ref{thm:improvement} and~\ref{thm:sepGen}. 
\item In Section 4 we prove Theorem~\ref{thm:improvement}.
\item In Section 5 we prove Theorems~\ref{thm:sepSp} and~\ref{thm:sepGen} and discuss the Hausdorff and box-counting dimension of $C_{n, D} \cap (C_{n, D} + \alpha)$. 
\end{enumerate} 
\section{\textbf{Background}} \label{section:Background} 

We recall definitions and facts from fractal geometry. There are multiple kinds of fractal dimensions used to measure the complexity of irregular geometric objects. This paper features two of the most popular, the Hausdorff dimension and the box-counting dimension, and also packing dimension. The definition of the box-couting dimension was given in Definition~\ref{def:boxcounting}. We recall the definitions of Hausdorff dimension and packing dimension in Euclidean space. All three dimensions could be defined for more general metric spaces. 

\begin{definition}\label{def:hausDim}
Let $X$ be a subset of $\mathbb{R}^{m}$. Given $s > 0$, the \textit{$s$-dimensional Hausdorff content of $X$} is the quantity
\begin{equation}
\mathcal{H}_{\infty}^{s}(X) := \inf\sum_{i}\diam(U_{i})^{s}
\end{equation}
where the infimum is taken over all countable covers $\{U_{i}\}$ of $X$ by any subsets of $\mathbb{R}^{m}$. 
We call $\dim_{H}X := \inf\{s\geq0:\mathcal{H}_{\infty}^{s}(X) = 0\}$ the \textit{Hausdorff dimension of $X$}. 
\end{definition}

\begin{definition}\label{def:deltaPacking}
Let $E\subset\mathbb{R}^{m}$. For $\delta > 0$, we call a countable (possibly finite) collection of disjoint balls, each with radius less than or equal to $\delta$ and with their center in $E$, a \textit{$\delta$-packing of $E$}. 
\end{definition}

\begin{definition}\label{def:packingDim}
Let $X$ be a subset of $\mathbb{R}^{m}$. Given positive numbers $s$ and $\delta$, and any subset $E\subset\mathbb{R}^{m}$, define the quantity
\begin{equation}
\mathcal{P}_{\delta}^{s}(E) := \sup\sum_{i=1}^{\infty}\diam(B_{i})^{s}
\end{equation}
where the supremum is taken over all $\delta$-packings of $E$. Let $\mathcal{P}_{0}^{s}(E) := \lim_{\delta\rightarrow0^{+}}P_{\delta}^{s}(E)$. The \textit{$s$-dimensional packing measure of $X$} is the quantity
\begin{equation}
\mathcal{P}^{s}(X) := \inf\bigg\{\sum_{i=1}^{\infty}\mathcal{P}_{0}^{s}(X_{i}): X \subset \cup_{i=1}^{\infty}X_{i}\bigg\}. 
\end{equation}
We call $\dim_{P}X := \inf\{s\geq0:\mathcal{P}^{s}(X) = 0\}$ the \textit{packing dimension of $X$}.
\end{definition}

For all bounded subsets of $X\subset\mathbb{R}^{m}$, we have 
\begin{equation}\label{eq:dimRel} 
\dim_{H}X \leq \dim_{P}X \leq \overline{\dim_{B}}X \;\; \text{and} \;\; \dim_{H}X \leq \underline{\dim_{B}}X \leq \overline{\dim_{B}}X.
\end{equation}
See section 3.4 of \cite{F90} for details. For self-similar sets, such as the middle third Cantor set, all four notions of dimension agree (corollary 3.3 in \cite{F97}). In particular, if a self-similar set is the attractor of an IFS that satisfies a type of separation condition, then not only do all the dimensions agree but they are easily computable. We recall one of those separation conditions. 

\begin{definition}\label{def:strongSep} 
Suppose $A$ is the attractor of an IFS given by $\mathcal{F} = \{f_{i}\}_{i=1}^{N}$. We say that $\mathcal{F}$ satisfies the \textit{strong separation condition} (SSC) if the images $f_{i_{1}}(A)$ and $f_{i_{2}}(A)$ are disjoint for every distinct pair $1 \leq i_{1}, i_{2} \leq N$. 
\end{definition}

\begin{theorem}[K. Falconer, \cite{F97}, corollary 3.3] \label{thm:strongSimDim} 
Suppose $A$ is the attractor of the set of similarities $\mathcal{F} = \{f_{i}\}_{i=1}^{N}$. For each $i$, let $r_{i}$ be the contraction ratio of $f_{i}$. If $\mathcal{F}$ satisfies the SSC, then $\dim_{B}A = \dim_{H}A = \dim_{P}A = s$ where $s$ is the unique positive solution to $\sum_{i=1}^{N}r_{i}^{s} = 1$. 
\end{theorem}

It is common to state a stronger version of this theorem (see theorem 9.3 of \cite{F90}). It replaces the SSC with a weaker notion called the \textit{open set condition} (OSC). An IFS given by $\{f_{i}:X\rightarrow X\}_{i=1}^{N}$ satisfies the OSC if there exists an open set $\mathcal{O}\subset X$ such that $\cup_{i=1}^{N}f_{i}(\mathcal{O}) \subset \mathcal{O}$. We also remark that the cited statements in \cite{F90} and \cite{F97} do not mention the packing dimension, but the first inequality in (\ref{eq:dimRel}) implies the assertion immediately.

\section{\textbf{Neighbouring Tiles}} \label{sec:neighbours} 

In this section we derive the bounds present in Theorem~\ref{thm:improvement} and Theorem~\ref{thm:sepGen}. We begin with two definitions. 

\begin{definition} 
Let $Y\subset\mathbb{C}$. We call any Gaussian integer $s$ that satisfies
\begin{equation} \label{eq:neighbourcondition}
Y \cap (Y+s) \neq \emptyset
\end{equation}
a neighbour of $Y$. 
\end{definition}

\begin{definition}\label{def:fundTile} 
Fix an integer $n \geq 2$ and set $b := -n+i$. We call the attractor of the IFS given by $\{f_{d}(x) = b^{-1}(x + d) : d\in\{0, 1, \ldots, n^{2}\}\}$ the \textit{$n$th fundamental tile} and denote it by $T_{n}$. Similarly, we call the attractor of the IFS given by $\{f_{d}(x) = b^{-1}(x+\delta) : d\in\{0, \pm1, \ldots, \pm n^{2}\}\}$ the \textit{$n$th extended tile} and denote it by $\mathcal{E}_{n}$. 
\end{definition}

Our goal is to find the real neighbours of $T_{n}$ and $\mathcal{E}_{n}$ for $n\geq2$. To this end, we prove the following lemma. 

\begin{lemma}\label{lem:reimbound} 
Fix an integer $n \geq 2$. Suppose $s$ is a neighbour of $T_{n}$. Then 
\begin{equation}
|\Re(s)-n\Im(s)| < 2.
\end{equation}
Moreover $|\Re(s)-n\Im(s)| < 3/2$ for $n\geq5$. 
\end{lemma}

\begin{proof}
Let $s$ be a neighbour of $T = T_{n}$. For convenience, we set $\alpha := \Re(s)$ and $\beta := \Im(s)$. By definition there exist sequences $(d_{j})_{j=1}^{\infty}$ and $(d_{j}^{'})_{j=1}^{\infty}$ with entries in $\{0, 1, \ldots, n^{2}\}$ such that 
\begin{equation}
s + \pi(d_{j})_{j=1}^{\infty} = \pi(d_{j}^{'})_{j=1}^{\infty}.
\end{equation}
Isolating for $s$ yields 
\begin{equation}
s = \pi(\delta_{j})_{j=1}^{\infty}
\end{equation}
where $\delta_{j}\in\{0, \pm 1, \ldots, \pm n^{2}\}$ for each $j$. 

We wish to estimate the difference between the real part of $s$ and $n$ times its imaginary part. We explicitly compute the first few terms of $s$ in terms of $n$. Observe that 
\begin{equation}
s = \frac{\delta_{1}}{b} +  \frac{\delta_{2}}{b^{2}} + \frac{\delta_{3}}{b^{3}} + \frac{\delta_{4}}{b^{4}} + \epsilon
\end{equation}
where $\epsilon$ denotes the tail $\sum_{j=5}^{\infty}\delta_{j}b^{-j}$.

Explicit computation yields
\begin{equation} \label{eq:neighbourreal} 
\alpha = \frac{-n}{n^{2}+1}\delta_{1} +  \frac{n^{2}-1}{(n^{2}+1)^{2}}\delta_{2} + \frac{-n^{3}+3n}{(n^{2}+1)^{3}}\delta_{3} + \frac{n^{4}-6n^{2}+1}{(n^{2}+1)^{4}}\delta_{4} + \Re(\epsilon) \\ 
\end{equation}
and 
\begin{equation} \label{eq:neighbourimag} 
\beta = \frac{-1}{n^{2}+1}\delta_{1} + \frac{2n}{(n^{2}+1)^{2}}\delta_{2} + \frac{-3n^{2}+1}{(n^{2}+1)^{3}}\delta_{3} + \frac{4n^{3}-4n}{(n^{2}+1)^{4}}\delta_{4} + \Im(\epsilon). \\
\end{equation}
Subtracting $n$ times (\ref{eq:neighbourimag}) from (\ref{eq:neighbourreal}) yields
\begin{equation} \label{eq:realimagdiff} 
\alpha - n\beta = \frac{-1}{n^{2}+1}\delta_{2} + \frac{2n}{(n^{2}+1)^{2}}\delta_{3} + \frac{-3n^{4}-2n^{2}+1}{(n^{2}+1)^{4}}\delta_{4} + \Re(\epsilon) - n\Im(\epsilon). 
\end{equation}
Recall that $\delta_{j}$ range from $-n^{2}$ to $n^{2}$. In order to bound $|\alpha - n\beta|$ with an expression that is only in terms of $n$, we maximize the sum of the first three terms of (\ref{eq:realimagdiff}) by choosing $\delta_{2} = \delta_{4} = -n^{2}$ and $\delta_{3} = n^{2}$ and estimate the absolute value of the last two terms using the bound $|Re(\epsilon) - n\Im(\epsilon)| \leq (n+1)\sum_{j=5}^{\infty}n^{2}|b|^{-j} = (n+1)(|b|+1)/|b|^{4}$. This results in the inequality
\begin{equation} \label{eq:boundnonly} 
|\alpha - n\beta| \leq \frac{n^{2}}{n^{2}+1} + \frac{2n^{3}}{(n^{2}+1)^{2}} + \frac{3n^{6}+2n^{4}-n^{2}}{(n^{2}+1)^{4}} + \frac{(n+1)(\sqrt{n^{2}+1}+1)}{(n^{2}+1)^{2}}
\end{equation}
Direct computation with $n = 3, 4, 5, 6$ yields bounds less than $1.85, 1.63, 1.5,$ and $1.41$.
For $n \geq 7$, observe that we can respectively bound each term of (\ref{eq:boundnonly}) by the following sequences. 
\begin{align}
\frac{n^{2}}{n^{2}+1} &< 1, \label{eq:firstbound}\\ 
\frac{2n^{3}}{(n^{2}+1)^{2}} &< \frac{2}{n}, \\
\frac{3n^{6}+2n^{4}-n^{2}}{(n^{2}+1)^{4}} &< \frac{5}{n^{2}}, \\
 \frac{(n+1)(\sqrt{n^{2}+1}+1)}{(n^{2}+1)^{2}} &< \frac{5}{n^{2}}. \label{eq:lastbound}
\end{align}
Consider the sum of all the sequences on the right hand side of a ``$<$" sign from (\ref{eq:firstbound}) to (\ref{eq:lastbound}). The sum is a strictly decreasing sequence and at $n = 7$ is less than $1.49$. This completes the proof. 
\end{proof}

\begin{corollary}\label{cor:realNeighbours} 
Fix an integer $n\geq 2$. The set of real neighbours of $\mathcal{E}_{n}$ is
\begin{itemize}
\item[(i)] $\{0, \pm1, \pm2\}$ when $n=1$ or $n\geq5$ .
\item[(ii)] $\{0, \pm1, \pm2, \pm3\}$ when $n=2, 3,$ or $4$.
\end{itemize}
Additionally, the set of real neighbours of $T_{n}$ is $\{0, \pm 1\}$ for all $n\geq1$. 
\end{corollary}

\begin{proof}
The set of neighbours of $T_{n}$ when $n = 1, 2$ and the set of neighbours of $\mathcal{E}_{n}$ when $n = 1, 2, 3$ or $4$ can be explicitly computed using the neighbour finding algorithm found in \cite{ST03}. We prove the remaining cases. 

It follows immediately from Lemma~\ref{lem:reimbound} that the real neighbours of $T_{n}$ must be in $\{0, \pm 1\}$. To see the converse, it can be verified by direct computation that 
\begin{align}
1 + 0.(2n-1)\overline{((n-1)^{2}+1)0} &= 0.0\overline{0((n-1)^{2}+1)}, \label{eq:examplecomputation}\\
-1 + 0.0\overline{0((n-1)^{2}+1)} &= 0.(2n-1)\overline{((n-1)^{2}+1)0},
\end{align}
where the bar indicates infinite repetition of those digits in the order presented. We include the verification of (\ref{eq:examplecomputation}).

Let 
\begin{align}
z_{1} := 0.\overline{((n-1)^{2}+1)0}, \\
z_{2} := 0.\overline{0((n-1)^{2}+1)}. 
\end{align}
Since the sequences of digits are periodic with period two and letting $b := -n+i$, 
\begin{align}
b^{2}z_{1} - z_{1} &= ((n-1)^{2}+1)b, \\
b^{2}z_{2} - z_{2} &=  (n-1)^{2} + 1. 
\end{align}
From these equations it is possible to solve for $z_{1}$ and $z_{2}$ explicitly in terms of $n$. We obtain
\begin{align}
z_{1} &=  \frac{((n-1)^{2}+1)(-n+i)}{(n^{2}-2-2ni)} \\
z_{2} &=  \frac{((n-1)^{2}+1)}{(n^{2}-2-2ni)}.
\end{align}
We wish to show that 
\begin{equation}\label{eq:goal}
\frac{((n-1)^{2}+1)}{(n^{2}-2-2ni)} + \frac{2n-1}{-n+i} + 1 =  \frac{((n-1)^{2}+1)}{(n^{2}-2-2ni)(-n+i)}
\end{equation}
Observe that on the left hand side of (\ref{eq:goal}), after bringing it under a common denominator, the numerator is 
\begin{equation}
((n-1)^{2}+1)(-n+i) + (2n-1)(n^{2}-2-2ni)+(-n+i)(n^{2}-2-2ni). 
\end{equation}
It now can be seen from the right hand side of (\ref{eq:goal}) that it is sufficient to show that 
\begin{equation}
(2n-1)(n^{2}-2-2ni)+(-n+i)(n^{2}-2-2ni) = (1 + n - i)((n-1)^{2}+1). 
\end{equation}
We conclude with
\begin{align}
&(2n-1)(n^{2}-2-2ni)+(-n+i)(n^{2}-2-2ni) \\
&= (n - 1 + i)(n^{2}-2-2ni) \\
&= ((n-1)^{2}+1)(n^{2}-2-2ni)/(n - 1 - i) \\
&= ((n-1)^{2}+1)(n + 1 - i).
\end{align}

Now we consider the $n$th extended tile. If $p$ is a neighbour of $\mathcal{E}_{n}$, then $p = \pi(p_{j})_{j=1}^{\infty}$ where $p_{j}\in\{0, \pm 1, \ldots, \pm (2n^{2}-1), \pm (2n^{2})\}$. It follows that the function bounding $|\Re(s)-\Im(s)|$ in (\ref{eq:boundnonly}) merely needs to be doubled in order to bound $|\Re(p)-n\Im(p)|$. Therefore $|\Re(p)-n\Im(p)| < 3$ and the real neighbours of $\mathcal{E}_{n}$ are contained in $\{0, \pm 1, \pm 2\}$.  To see that $2$ and $-2$ are neighbours of $\mathcal{E}_{n}$ for $n\geq5$, it can be verified directly that $0.0\overline{(-n^{2})(n-2)^{2}} = 2.(4n-2)\overline{(n-2)^{2}(-n^{2})}$. 
\end{proof}

We end this section by demonstrating a further application of Lemma~\ref{lem:reimbound}. Gilbert gave a proof of the following result in \cite{G82}. 
\begin{theorem}[W. Gilbert, \cite{G82}, proposition 1]\label{thm:neighbourset}
A Gaussian integer $s$ is a neighbour of $T_{n}$ if and only if
\begin{itemize}
\item[(i)] $s\in\{0, \pm1, \pm (n-1+i), \pm (n+i)\}$ and $n\geq3$.
\item[(ii)] $s\in\{0, \pm1, \pm (1+i), \pm (2+i), \pm i, \pm (2+2i)\}$ and $n=2$. 
\end{itemize}
\end{theorem}
Gilbert used this result to derive the rules governing radix expansions in base $(-n+i, \{0, 1, \ldots, n^{2}\})$ (see theorem 5 and theorem 8 of \cite{G82}). Our proof, by way of Lemma~\ref{lem:reimbound}, uses a different approach than that of Gilbert. 

We first point out the following simple observation. 

\begin{lemma}\label{lem:infinitewalk} 
Fix an integer $n\geq 2$. If $s$ is a neighbour of $T_{n}$, then $bs + \delta$ is a neighbour of $T_{n}$ for some $\delta\in\{0, \pm 1, \ldots, \pm n^{2}\}$. 
\end{lemma}

\begin{proof}
Again, there exist sequences $(d_{j})_{j=1}^{\infty}$ and $(d_{j}^{'})_{j=1}^{\infty}$ with entries in $\{0, 1, \ldots, n^{2}\}$ such that 
\begin{equation}
s + \pi(d_{j})_{j=1}^{\infty} = \pi(d_{j}^{'})_{j=1}^{\infty}. 
\end{equation}
This equation holds if and only if
\begin{equation}
bs + (d_{1}-d_{1}^{'}) + \pi(d_{j+1})_{j=1}^{\infty} = \pi(d_{j+1}^{'})_{j=1}^{\infty}. 
\end{equation}
This completes the proof. 
\end{proof}

We are now in position to prove Theorem~\ref{thm:neighbourset}. 

\begin{proof}[Proof of Theorem \ref{thm:neighbourset}] 
These case $n = 2$ can be computed explicitly using the neighbour finding algorithm in \cite{ST03}. We proceed assuming $n\geq 3$. Suppose $s$ is a neighbour of $T_{n}$. 

By Lemma \ref{lem:reimbound}, if $s$ is real, then $s\in\{0, \pm 1\}$. That lemma also implies that $s$ cannot be purely imaginary. If that were the case, $|n\Im(s)| < 2$ where $n\geq3$. This is impossible. 

We now claim that $|\Im(s)|$ cannot be larger than $1$. Let us denote $\Im(s)$ by $\beta$. Lemma $\ref{lem:reimbound}$ implies that $Re(s)$ is one of $n\beta - 1$, $n\beta$, or $n\beta+1$. Let us assume $\beta$ is greater than or equal to $2$. The magnitude of $s$ is therefore bounded below by $n\beta-1$.
On the otherhand, $|s| \leq n^{2}\sum_{j=1}^{\infty}|b|^{-j} = \sqrt{n^{2}+1} + 1$. We will argue that this quantity is strictly less $n\beta-1$ for $n\geq3$. 
Observe that for $n\geq3$
\begin{align}
\sqrt{n^{2}+1} + 1 &< n+2 \\
&\leq 2n-1 \\
&\leq n\beta-1.  
\end{align}
The second inequality holds since $n - 3 \geq 0$. This shows that when $\beta\geq2$, the number $s$ is not a neighbour of $T_{n}$. The case when $\beta$ is less than or equal to $-2$ is similar.

The only remaining cases to consider are $\beta = 1$ or $-1$. Lemma \ref{lem:reimbound} implies that $s$ is one of $\pm(n-1+i), \pm(n+i)$, and $\pm(n+1+i)$. 

By Lemma~\ref{lem:infinitewalk}, if $n+1+i$ is a neighbour of $T_{n}$, there must exist $\delta\in\{0, \pm 1, \ldots, \pm n^{2}\}$ such that $b(n+1+i) + \delta$ is also a neighbour of $T_{n}$. 
The set of neighbours of $T_{n}$ is a subset of $\{0, \pm1, \pm(n-1+i), \pm(n+i), \pm(n+1+i)\}$. We observe that $b(n+1+i) = -(n^{2}+n+1)+i$. We see that adding an element of $\{0, \pm 1, \ldots, \pm n^{2}\}$ (a real number) to this expression cannot result in any of $0$, $1$, and $-1$. To see that not a single neighbour is obtainable, we compute $\delta$ using the remaining potential neighbours of $T_{n}$. 
\begin{align}
(n+i) + n^{2}+n+1 - i &= n^{2} + 2n + 1, \\
(-n-i) + n^{2}+n+1 - i &= n^{2}+1-2i, \\
(n-1+i) + n^{2}+n+1 - i &= n^{2}+2n, \\
(-n+1-i) + n^{2}+n+1 - i &= n^{2} +2 -2i, \\
(n+1+i) + n^{2}+n+1 - i &= n^{2}+2n+2, \\
(-n-1-i) + n^{2}+n+1 - i &= n^{2} -2i.
\end{align}
All of these are either larger than $n^{2}$ or are not real and therefore are not in $\{0, \pm 1, \ldots, \pm n^{2}\}$. We conclude that $n+1+i$ is not a neighbour of $T_{n}$. Similarly, it can be shown that $-n-1-i$ is not a neighbour of $T_{n}$. 

We have shown that a neighbour of $T$ is an element of the set $\{0, \pm1, \pm (n-1+i), \pm (n+i)\}$. To see that the converse also holds, we show that there exists $t\in T\cap(T+s)$ for each $s$. It can be verified explicitly that

\begin{align}
1 + 0.(2n-1)\overline{((n-1)^{2}+1)0} &= 0.0\overline{0((n-1)^{2}+1)}, \\
-1 + 0.0\overline{0((n-1)^{2}+1)} &= 0.(2n-1)\overline{((n-1)^{2}+1)0}, \\
(n+i) + 0.n^{2}0\overline{0((n-1)^{2}+1)} &= 0.0(2n-1)\overline{((n-1)^{2}+1)0}, \\
(-n-i) + 0. 0(2n-1)\overline{((n-1)^{2}+1)0} &= 0.n^{2}0\overline{0((n-1)^{2}+1)},\\
(n-1+i) + 0.\overline{((n-1)^{2}+1)0} &= 0.\overline{0((n-1)^{2}+1)},\\
(-n+1-i) + 0.\overline{0((n-1)^{2}+1)} &= 0.\overline{((n-1)^{2}+1)0}.
\end{align}
The verification can be performed in the same way as in the proof of Corollary~\ref{cor:realNeighbours}.

\end{proof}

\section{\textbf{An Application to Box-Counting Dimension}}\label{sec:boxIntersect} 

In this section we prove Theorem~\ref{thm:improvement}. Let us recall the definition of the box-counting dimension. 

\begin{definition} 
Let $X$ be a bounded subset of $\mathbb{R}^{m}$. Given $\delta>0$, we let $N_{\delta}(X)$ denote the smallest number of sets of diameter $\delta$
needed to cover $X$. The \textit{upper box-counting dimension} and the \textit{lower box-counting dimension of $X$} are
\begin{align}
\overline{\dim}_{B}X &:= \limsup_{\delta\rightarrow0}\frac{\log N_{\delta}(X)}{-\log\delta}, \\
\underline{\dim}_{B}X &:= \liminf_{\delta\rightarrow0}\frac{\log N_{\delta}(X)}{-\log\delta}, 
\end{align}
respectively. If these quantities are equal, then that value is \textit{the box-counting dimension of $X$} and is denoted by $\dim_{B}X$. 
\end{definition}

It is known that we can replace the function $N_{\delta}$ with the function that counts the number of $\delta$-mesh cubes that intersect $X$ and still capture the upper and lower box-counting dimensions (see 3.1 in \cite{F90}). This is the collection of cubes $\{[k_{1}\delta, (k_{d}+1)\delta] \times \cdots \times [k_{d}\delta, (k_{d}+1)\delta] : k_{j}\in\mathbb{Z}\}$. In the plane, it is a collection of squares. For our purposes, we wish to count translations of scalings of the $n$th fundamental tile $T_{n}$ (see Definition~\ref{def:fundTile}). 

\begin{definition} 
Let $A$ be the attractor of an iterated function system $\{f_{1}, f_{2}, \ldots, f_{m}\}$ and let $k$ be a positive integer. A $k$-tile of $A$ is any set of the form
$A_{j_{1}, j_{2}, \ldots, j_{k}} := (f_{j_{1}}\circ f_{j_{2}} \circ \cdots \circ f_{j_{k}})(A)$. 
\end{definition}

For $n\geq 2$, the $k$-tiles of $T = T_{n}$ are the sets $T_{d_{1}, \ldots, d_{k}} = \{0.d_{1}d_{2}\ldots d_{k} + b^{-k}t : t\in T\}$, where $d_{1}, d_{2}, \ldots, d_{k}\in \{0, 1, 2, \ldots, n^{2}\}$. The following lemma, Lemma~\ref{lem:boxtile}, allows us to trade counting boxes with counting $k$-tiles in order to compute the upper and lower box-counting dimensions.

\begin{lemma}\label{lem:boxtile}
Fix an integer $n\geq 2$. Let $F$ be a nonempty subset of $T_{n}$. For a fixed integer $k\geq1$, let $N_{k}(F)$ denote the number of $k$-tiles of $T_{n}$ that intersect $F$. Then 
\begin{align}
\overline{\dim}_{B}F &= \limsup_{k\rightarrow\infty}\frac{\log N_{k}(F)}{k\log{|b|}}, \\
\underline{\dim}_{B}F &= \liminf_{k\rightarrow\infty}\frac{\log N_{k}(F)}{k\log{|b|}}. 
\end{align}
\end{lemma}

The proof of this result can be found in \cite{PS21}. Computing $N_{k}(F)$ can be made easier if $k$-tiles are disjoint. We provide a sufficient condition. 

\begin{lemma}\label{lem:sepOne} 
Fix an integer $n\geq 2$. Let $D\subset\{0, 1, \ldots, n^{2}\}$ satisfy $|d-d^{'}|\neq1$ for all $d\in D$. The $k$-tiles $T_{d_{1}, \ldots, d_{k}}$ and $T_{d_{1}^{'}, \ldots, d_{k}^{'}}$ of $T_{n}$ are disjoint whenever $d_{j}\neq d_{j}^{'}$ for some $1\leq j \leq k$ and $d_{j}, d_{j}^{'}\in D$ for each $j$. 
\end{lemma}

\begin{proof}
Suppose that $(d_{1}, d_{2}, \ldots, d_{k})$ and $(d_{1}^{'}, d_{2}^{'}, \ldots, d_{k}^{'})$ are distinct tuples with $d_{j}, d_{j}^{'}\in D$ for each $j$. 
Let $J$ be the smallest index in the set $\{1, 2, \ldots, k\}$ at which the specified digits differ. Suppose that the intersection of the two $k$ tiles is nonempty. By definition, there exists sequences $(d_{k+j})_{j\geq1}$ and $(d_{k+j}^{'})_{j\geq1}$ with $d_{k+j}, d_{k+j}^{'}\in \{0, 1, \ldots, n^{2}\}$ for each $j$ such that $\pi(d_{j})_{j=1}^{\infty} = \pi(d_{j}^{'})_{j=1}^{\infty}$. Multiplying by $b^{J}$ and then subtracting $d_{1}b^{J-1}+\cdots+d_{J}$ on both sides yields $\pi(d_{J+j})_{j=1}^{\infty} = (d_{J}^{'}-d_{J}) + \pi(d_{J+j}^{'})_{j=1}^{\infty}$ and thus $d_{J}^{'}-d_{J}$ is a neighbour of $T_{n}$ where $|d_{J}-d_{J}^{'}| > 1$. This contradicts Corollary~\ref{cor:realNeighbours}. 
\end{proof}

Theorem~\ref{thm:boundedDigitFormula} and Theorem~\ref{thm:unboundedDigitFormula} are extensions of theorem 7.4 in \cite{PS21}. Recall that $C_{n, D}$ denotes the restricted digit set generatedy by $(n, D)$ (see Definition~\ref{def:restrictedDigitSet}). 

\begin{theorem}\label{thm:boundedDigitFormula} 
Fix $n\geq2$ and suppose $D\subset\{0, 1, \ldots, \lfloor n^{2}/2 \rfloor\}$ is such that $|\delta-\delta^{'}| \neq 1$ for all $\delta, \delta^{'}\in \Delta := D - D$. If $\alpha = \pi(\alpha_{j})_{j=1}^{\infty}$ such that $\alpha_{j}\in \Delta$, then 
\begin{equation}
\underline{\dim}_{B}(C_{n, d} \cap (C_{n, D} + \alpha)) = \liminf_{k\rightarrow\infty} \frac{\log{M_{k}(\alpha)}}{k\log{|b|}}
\end{equation}
where $M_{k}(\alpha) := \prod_{j=1}^{k}|D\cap(D+\alpha_{j})|$. 
\end{theorem}

\begin{remark}
The result in \cite{PS21}, theorem 7.4, requires that $|\delta - \delta^{'}| > n$ for all $\delta, \delta^{'}\in\Delta$ as opposed to the separation condition of $1$. We also mention that the lower box-counting dimension could be replaced with the upper box-counting dimension. 
\end{remark}

\begin{proof}
Let $C(\alpha) := C_{n, D} \cap (C_{n, D} + \alpha)$. According to Lemma \ref{lem:boxtile}, it suffices to prove that $\lambda M_{k}(\alpha) \leq N_{k}(C(\alpha)) \leq \rho M_{k}(\alpha)$ for contants $\lambda, \rho $. This will hold if we show that the collection of $k$-tiles $T_{d_{1}, d_{2}, \ldots, d_{k}}$, where $d_{j}\in D\cap(D+\alpha_{j})$ covers $C(\alpha)$ and each $k$-tile in the cover intersects $C(\alpha)$. 

Let $T_{d_{1}, d_{2}, \ldots, d_{k}}$ be a $k$-tile with $d_{j}\in D\cap(D+\alpha_{j})$ for $j=1, 2, \ldots, k$. For each $j$, there exists $d_{j}^{'}\in D$ such that equation $d_{j} = d_{j}^{'} + \alpha_{j}$ holds. In general, $\alpha_{j} = d_{j}^{''} - d_{j}^{'''}$ with $d_{j}^{''}, d_{j}^{'''}\in D$ for each $j$. Consider the number $z := \pi(z_{j})_{j=1}^{\infty}$ such that $z_{j} = d_{j}$ for all $1\leq j \leq k$ and $z_{j} = d_{j}^{''}$ for all $j>k$. Immediately we see that $z$ is an element of $C_{n, D}$. Since the number given by $0.d_{1}^{'}d_{2}^{'}\ldots d_{k}^{'}d_{k+1}^{'''}d_{k+2}^{'''}\ldots$ is also an element of $C_{n, D}$, it follows that $z \in C_{n, D} +\alpha$. Therefore $T_{d_{1}, d_{2}, \ldots, d_{k}}$ intersects $C(\alpha)$. We conclude that $ M_{k}(\alpha) \leq N_{k}(C(\alpha))$. 

Now we argue that for any positive integer $k$ the $C(\alpha)$ is covered by the $k$-tiles which specifiy digits in $D\cap(D+\alpha_{j})$ for each $j=1, 2, \ldots, k$. Let $z\in C(\alpha)$. This means there exists sequences $(d_{j})_{j=1}^{\infty}$ and $(d_{j}^{'})_{j=1}^{\infty}$ such that 
\begin{equation} \label{eq:intersectionelement}
\pi(d_{j})_{j=1}^{\infty} = \pi(d_{j}^{'}+\alpha_{j})_{j=1}^{\infty}
\end{equation}
Recall that $\alpha_{j} = d_{j}^{''} - d_{j}^{'''}$ for each $j$. Adding $\pi(d_{j}^{'''})_{j=1}^{\infty}$ to both sides of (\ref{eq:intersectionelement}) yields 
\begin{equation}
\pi(d_{j}+d_{j}^{'''})_{j=1}^{\infty}= \pi(d_{j}^{'}+d_{j}^{''})_{j=1}^{\infty}. 
\end{equation}

Since $d\leq n^{2}/2$ for all $d\in D$, and $d_{j}, d_{j}^{'}, d_{j}^{''}$, and $d_{j}^{'''}$ are all elements of $D$ for all $j$, their pairwise sums are in $\{0, 1, \ldots, n^{2}\}$. Furthermore, any separation condition on the elements of $\Delta$ holds if and only the same condition holds for elements of $D+D$. This is because $(a+b) - (c+d) = (a-c) - (d-b)$ for any collection of integers $a, b, c, d$. It follows from Lemma~\ref{lem:sepOne} that $d_{j} + d_{j}^{'''} = d_{j}^{'} + d_{j}^{''}$. In particular, $d_{j} = d_{j}^{'} + x_{j}$ and we conclude that $z\in T_{d_{1}, d_{2},\ldots, d_{k}}$ for any $k$ where $d_{j}\in D\cap(D+\alpha_{j})$. It follows from Theorem~\ref{thm:neighbourset} that each of these $k$-tiles intersects at most $9$ $k$-tiles of $T$. This yields $N_{k}(C(\alpha)) \leq 9M_{k}(\alpha)$. 
\end{proof}

The following example is an application of Theorem~\ref{thm:boundedDigitFormula} to a case that is not covered by theorem 7.4 in \cite{PS21}. 

\begin{example} \label{ex:boxEx}
Let $b=-3+i$. The subsets of $\{0, 1,\ldots, 9\}$ for which $d\leq9/2$ are subsets of $\{0, 1, 2, 3, 4\}$. The constraint that every pair of elements $a, a^{'}\in\Delta$ satisfies $|a-a^{'}|\neq1$ yields the following subsets of $\{0, 1, 2, 3, 4\}$ that are not singletons: $\{0, 2\}$, $\{0, 3\}$, $\{0, 4\}$, $\{1, 3\}$, $\{1, 4\}$, $\{2, 4\}$, and $\{0, 2, 4\}$. If $D$ is a singleton, then so is $C$. The box-counting dimension of $C_{n, D}$, and consequently $C(\alpha)$ for any $\alpha$, is zero in that case.  

Suppose we choose $D = \{0, 4\}$. Therefore $\Delta$ is equal to $\{-4, 0, 4\}$. The intersection of $C_{3, \{0, 4\}}$ and its translation by $\alpha = \frac{-28+24i}{-19+26i} = 0.\overline{-404}$ is nonempty. We compute its box-counting dimension. Since $|D\cap(D-4)| = |D\cap(D+4)| = 1$ and $|D| = 2$, it follows that
\begin{equation}
M_{k}(\alpha) = \begin{cases} 2^{k/3} &\;\text{if}\;k \equiv 0 \mod{3},\\
 2^{(k-1)/3} &\;\text{if}\; k \equiv 1 \mod{3},\\
 2^{(k+1)/3} &\;\text{if}\;k \equiv 2 \mod{3}.
\end{cases}
\end{equation}
Therefore $2^{(k-1)/3} \leq M_{k}(\alpha) \leq 2^{(k+1)/3}$ for all $k$. In particular, 
\begin{equation}
\frac{(k-1)\log{2}}{3k\log{10}} \leq \frac{\log{M_{k}(\alpha)}}{k\log{10}}\leq \frac{(k+1)\log{2}}{3k\log{10}}.
\end{equation}
By Theorem~\ref{thm:boundedDigitFormula}, we conclude that $\dim_{B}(C(\alpha)) = \frac{\log{2}}{3\log{10}}$.
\end{example}

It is possible to remove the bound of $n^{2}/2$ when the elements of $D-D$ satisfy a larger separation condition.

Fix a positive integer $n\geq 2$ and let $D$ be a subset of $\{0, 1, \ldots, n^{2}\}$. Recall that the set $\Delta := D-D$ is called \textit{sparse} if for all $\delta, \delta^{'}\in\Delta$ either $|\delta-\delta^{'}| > 2$ when $n\geq5$ or $|\delta-\delta^{'}| > 3$ when $n = 2, 3$ or $4$. 

The following lemma is analogous to Lemma~\ref{lem:sepOne}, but addresses the $n$th extended tile $\mathcal{E}_{n}$ (see Definition~\ref{def:fundTile}) rather than the $n$th fundamental tile $T_{n}$. 
\begin{lemma} \label{lem:extsep} 
Fix an integer $n\geq2$ and suppose $D\subset\{0, 1, \ldots, n^{2}\}$ is chosen such that $\Delta:= D-D$ is sparse. The $k$-tiles $\mathcal{E}_{\delta_{1}, \ldots, \delta_{k}}$ and $\mathcal{E}_{\delta_{1}^{'}, \ldots, \delta_{k}^{'}}$ are disjoint whenever $\delta_{j} \neq \delta_{j}^{'}$ for some index $1\leq j\leq k$ where $\delta_{j}, \delta_{j}^{'} \in \Delta$ for each $j$. 
\end{lemma}

\begin{proof} 
Suppose $n \geq 5$ and  $|\delta-\delta^{'}| > 2$ for all distinct $\delta, \delta^{'}\in \Delta$. Suppose $\mathcal{E}_{\delta_{1}, \ldots, \delta_{k}} \cap \mathcal{E}_{\delta_{1}^{'}, \ldots, \delta_{k}^{'}}$ is not empty. We can deduce in the same way as in the proof of Lemma~\ref{lem:sepOne} that an integer with magnitude greater than $2$ is a neighbour of $\mathcal{E}_{n}$. This contradicts Corollarly~\ref{cor:realNeighbours}. Similarly, if $n = 2, 3,$ or $4$ and $|\delta-\delta^{'}| > 3$ and we assume that the intersection is nonempty, we also obtain a contradiction with Corollary~\ref{cor:realNeighbours}.
\end{proof}

\begin{theorem}\label{thm:unboundedDigitFormula} 
Fix and integer $n\geq2$ and suppose $D\subset\{0, 1, \ldots, n^{2}\}$ is chosen such that $\Delta:= D-D$ is sparse. If $\alpha = \pi(\alpha_{j})_{j=1}^{\infty}$ such that $\alpha_{j}\in \Delta$, then 
\begin{equation}
\underline{\dim}_{B}(C_{n, D}\cap(C_{n, D} + \alpha)) = \liminf_{k\rightarrow\infty} \frac{\log{M_{k}(\alpha)}}{k\log{|b|}}
\end{equation}
where $M_{k}(\alpha) := \prod_{j=1}^{k}|D\cap(D+\alpha_{j})|$. 
\end{theorem}

This extends the application of the formula to sets $C_{n, D}$ that are not covered by Theorem~\ref{thm:boundedDigitFormula}. For example, for those sets defined using $n = 3$ and $D = \{0, 4, 8\}$. Theorem~\ref{thm:boundedDigitFormula} does apply to some cases that Theorem~\ref{thm:unboundedDigitFormula} does not. For example, $n = 3$ and $D = \{0, 3\}$. 

\begin{proof}
Let $C(\alpha) := C_{n, D} \cap (C_{n, D} + \alpha)$. The same argument used in the proof of Theorem~\ref{thm:boundedDigitFormula} can be used to show that $M_{k}(\alpha) \leq N_{k}(C(\alpha))$ and so we omit it. Let us skip to showing that every $z \in C(\alpha)$ is contained in a $k$-tile for any positive integer $k$ whose defining digits are in $D\cap(D+\alpha_{j})$ for each $j = 1, 2, \ldots, k$. 

For $z\in C_{n, D}$, there exists sequences $(d_{j})_{j\geq1}$ and $(d_{j}^{'})_{j\geq1}$ such that 
\begin{equation} \label{eq:intersectionelement2}
z = \pi(d_{j})_{j=1}^{\infty} = \pi(d_{j}^{'}+\alpha_{j})_{j=1}^{\infty}. 
\end{equation}
Substracting $\pi(d_{1}^{'})_{j=1}^{\infty}$ from the equation and obtain
\begin{equation}
\pi(d_{j}-d_{j}^{'})_{j=1}^{\infty} = \pi(\alpha_{j})_{j=1}^{\infty} . 
\end{equation}

By Lemma~\ref{lem:extsep}, it follows that $d_{j} = d_{j}^{'} + \alpha_{j}$. It immediately follows that $z$ is contained in $T_{d_{1}, d_{2}, \ldots, d_{k}}$. As in the proof of Theorem~\ref{thm:boundedDigitFormula} this implies $N_{k}(C(\alpha)) \leq 9M_{k}(\alpha)$. 
\end{proof}

The conclusions present in Theorem~\ref{thm:improvement} now follow as corollaries of Theorem~\ref{thm:boundedDigitFormula} and Theorem~\ref{thm:unboundedDigitFormula}. We include the proof from \cite{PS21} for completeness. For a fixed integer $n\geq 2$ and a subset $D\subset\{0, 1, \ldots, n^{2}\}$, set $C(\alpha) := C_{n, D} \cap (C_{n, D} + \alpha)$. Recall that $F_{n, D}$ denotes the fundamental set of translations $\{\alpha\in\mathbb{C} : C(\alpha) \neq \emptyset\}$. Recall the function $\Phi_{n, D} : \{\alpha\in F_{n, D}: \dim_{B}C(\alpha) \;\; \text{exists}\} \rightarrow [0, \dim_{B}C_{n, D}]$ given by $\alpha \mapsto \dim_{B}C(\alpha).$ 

\begin{corollary}\label{cor:boxDense} 
Fix an integer $n \geq 2$. Suppose either that $D \subset \{0, 1, \ldots, \lfloor n^{2}/2 \rfloor\}$ satisfies $|\delta - \delta^{'}| \neq 1$ for all $\delta, \delta^{'}\in\Delta := D-D$ or that $D\subset\{0, 1, \ldots, n^{2}\}$ is chosen such that $\Delta$ is sparse. 
The level sets of $\Phi_{n, D}$ are dense in $F_{n, D}$.
\end{corollary}

\begin{proof}
Let $\alpha \in F_{n, D}$ be given. We construct $\beta$ in the domain of $\Phi_{n, D}$ such that $\Phi_{n, D}(\beta) = \lambda\dim_{B}C_{n, D}$ where $\lambda \in [0, 1]$. The $\alpha = \pi(\alpha_{j})_{j=1}^{\infty}$ where $\alpha_{j}\in\Delta$. For any radius $r>0$, there exists an index $m$ at which any complex number of the form $\beta = 0.\alpha_{1}\alpha_{2}\ldots \alpha_{m}\beta_{m+1}\beta_{m+2}\ldots$ with $\beta_{j}\in\Delta$ is within distance $r$ of $\alpha$. We now describe how to choose the $\beta_{j}$ such that $\beta$ satisfies the desired properties. 

Suppose $0 < \lambda < 1$. There exists a sequence of integers $h_{j}$ such that $h_{j} \leq j\lambda < h_{j} + 1$. Since $\lambda < 1$ we see that $(j+1)\lambda < j\lambda + 1 < (h_{j}+1)+1$. It follows that either $h_{j+1} = h_{j}$ or $h_{j+1} = h_{j} + 1$. Let $d_{\text{max}}$ and $d_{\text{min}}$ denote the maximum and minimum of $D$ respectively. For all $j > m$, let 
\begin{equation}
\beta_{j} = \begin{cases}
d_{\text{max}} - d_{\text{min}} &\;\text{if}\; h_{j} = h_{j-1}, \\
0 &\;\text{if}\; h_{j} = h_{j-1} + 1. 
\end{cases}
\end{equation}

It follows that $|D\cap(D+\beta_{j})|$ is equal to either $1$ or $|D|$ for $j>m$. The key observation here is that in either case this is $|D|^{h_{j}-h_{j-1}}$. 

We can compute the lower box-counting dimension of $C(\beta) := C_{n, D} \cap (C_{n, D} + \beta)$ using either Theorem \ref{thm:boundedDigitFormula} or Theorem \ref{thm:unboundedDigitFormula}. Directly, 
\begin{align}
\underline{\dim}_{B}(C(\beta)) &= \liminf_{k\rightarrow\infty}\frac{\log{M_{k}(\beta)}}{k\log{|b|}} \\
&= \liminf_{k\rightarrow\infty}\frac{\log{\prod_{j=1}^{m}|D\cap(D+\alpha_{j})|}}{k\log{|b|}} + \liminf_{k\rightarrow\infty}\frac{\log{\prod_{j=m+1}^{k}|D\cap(D+\beta_{j})|}}{k\log{|b|}} \\
&= 0 + \liminf_{k\rightarrow\infty}\frac{\log{\prod_{j=m+1}^{k}|D|^{h_{j}-h_{j-1}}}}{k\log{|b|}} \\
&=  \liminf_{k\rightarrow\infty}\frac{(h_{k}-h_{m})\log{|D|}}{k\log{|b|}} \\
&= \alpha\dim_{B}C
\end{align}
We can see now that the limit infimum is in fact a limit. If $\lambda = 0$, then choose $\beta_{j} = d_{\text{max}} - d_{\text{min}}$ for all $j > m$. If $\lambda = 1$, then choose $\beta_{j} = 0$ for all $j > m$.
\end{proof}

It immediately follows that the function $\Phi_{n, D}$ is discontinuous everywhere on its domain.

We end this section by establishing versions of Corollary~\ref{cor:boxDense} for Hausdorff dimension and packing dimension. See Definition~\ref{def:hausDim} and Definition~\ref{def:packingDim} for details. 

\begin{definition} 
Fix an integer $n\geq2$ and suppose $D\subset\{0, 1, \ldots, n^{2}\}$. Let $C(\alpha) := C_{n, D} \cap (C_{n, D} + \alpha)$ for $\alpha \in \mathbb{C}$. We define the functions
\begin{equation}
 \begin{split}
    &\Psi_{n, D} : F_{n, D} \rightarrow [0, \dim_{H}C_{n, D}],\\
    &\Theta_{n, D} : F_{n, D} \rightarrow [0, \dim_{P}C_{n, D}],
  \end{split}
\quad
  \begin{split}
    &\Psi_{n, D}(\alpha) = \dim_{H}C(\alpha),\\
    &\Theta_{n, D}(\alpha) = \dim_{P}C(\alpha).
  \end{split}
\end{equation}
\end{definition}

Under the conditions of Corollary~\ref{cor:boxDense}, we can show that the level sets of these functions are also dense in $F_{n, D}$. 

\begin{definition} 
Let $(n_{k})_{k=1}^{\infty}$ be a sequence of positive integers, and let $(R_{k})_{k=1}^{\infty}$ be a sequence of vectors $(c_{k, 1}, c_{k, 2}, \ldots, c_{k, n_{k}})$ such that $0 < c_{k, j} < 1$ for all $k\geq1$ and $1 \leq j \leq n_{k}$. Let $J\subset\mathbb{R}^{m}$ be a compact set with nonempty interior. Let $\mathcal{F}$ denote a set of subsets of $\mathbb{R}^{m}$ indexed by $W := \{\emptyset\}\cup(\cup_{k=1}^{\infty}W_{k})$ where $W_{k} := \{(j_{1}, j_{2}, \ldots, j_{k}) : 1\leq j_{\ell} \leq n_{\ell}, 1\leq\ell\leq k\}$. Given two words $\sigma$ and $\tau$, we denote their concatenation by $\sigma * \tau$. 

The set $\mathcal{F}$ is said to \textit{satisfy the Moran structure given by $((n_{k})_{k=1}^{\infty}, (R_{k})_{k=1}^{\infty}, J)$} if it satisfies the following four conditions:

\begin{itemize} 
\item[(i)] $J_{\emptyset} = J$.
\item[(ii)] For any $\sigma\in W$, there exists a similarity $S_{\sigma}:\mathbb{R}^{m}\rightarrow\mathbb{R}^{m}$ such that $S_{\sigma}(J) = J_{\sigma}$. 
\item[(iii)] For any $k\geq 0$ and $\sigma\in W_{k-1}$, the interiors of $J_{\sigma*j}$ and $J_{\sigma*\ell}$ are disjoint for all $1 \leq j, \ell \leq n_{k}$ such that $j \neq \ell$. 
\item[(iv)]  For any $k\geq 0$ and $\sigma\in W_{k-1}$, $1 \leq j \leq n_{k}$, $\frac{\diam(J_{\sigma*j})}{\diam(J_{\sigma})} = c_{k, j}$. 
\end{itemize}

For $\mathcal{F}$ satisfying the Moran structure, set $E_{k} = \cup_{\sigma\in W_{k}}J_{\sigma}$ and $E = \cap_{k=1}^{\infty}E_{k}$. We call $E$ the \textit{Moran set associated with the collection $\mathcal{F}$}. 
\end{definition}

It is known that the Hausdorff dimension and packing dimension of a Moran set can be expressed in terms of a sequence derived from the ``contraction coefficients" $c_{k, j}$. 

\begin{theorem}[Hua S., Rao H., Wen Z., Wu J., \cite{HRWW00}, theorem 1.1]\label{thm:moranDim} 
Suppose $E$ is the Moran set associated with a collection $\mathcal{F}$ satisfying the Moran structure given by  $((n_{k})_{k=1}^{\infty}, (R_{k})_{k=1}^{\infty}, J)$. For each $k$, let $s_{k}$ be the solution to the equation $\sum_{\sigma\in W_{k}}c_{\sigma}^{s_{k}} = 1$ where $c_{\sigma} := \Pi_{\ell = 1}^{k}c_{\ell, j_{\ell}}$ for $\sigma = (j_{1}, j_{2}, \ldots, j_{k})$. If $\inf_{i, j}c_{i, j} > 0$, then
\begin{align}
\dim_{H}E &= \liminf_{k\rightarrow\infty}s_{k}, \\
\dim_{P}E &= \limsup_{k\rightarrow\infty}s_{k}. 
\end{align}
\end{theorem}

This is useful because the sets of the form $C_{n, D} \cap (C_{n, D} + \alpha)$ are Moran sets under conditions on $n, D,$ and $\alpha$. This allows us to state the following result. 

\begin{theorem} \label{thm:extDense} 
Fix an integer $n \geq 2$. Suppose either that $D \subset \{0, 1, \ldots, \lfloor n^{2}/2 \rfloor\}$ satisfies $|\delta - \delta^{'}| \neq 1$ for all $\delta, \delta^{'}\in\Delta := D-D$ or that $D\subset\{0, 1, \ldots, n^{2}\}$ is chosen such that $\Delta$ is sparse. The level sets of $\Psi_{n, D}$ and $\Theta_{n, D}$ are respectively dense in $F_{n, D}$.
\end{theorem}

\begin{proof}
Let $b:= -n+i$. We begin by arguing that the set $C(\alpha) := C_{n, D} \cap (C_{n, D} + \alpha)$ is a Moran set whenever $\alpha\in F_{n, D}$. It follows from the definition of $F_{n, D}$ that $\alpha = \pi(\alpha_{j})_{j=1}^{\infty}$ where $\alpha_{j}\in\Delta$ for each $j$. Choose $(n_{k})_{k=1}^{\infty} = (|D\cap(D+\alpha_{k})|)_{k=1}^{\infty}$, $(R_{k})_{k=1}^{\infty}$ such that the entries of every vector are $|b|^{-1}$, and $J = C_{n, D}$. Let $\mathcal{F}_{k}$ denote the set of cylinders of the form $\pi(\{d_{1}\} \times \cdots \times \{d_{k}\} \times D^{\mathbb{N}})$ where $d_{j}\in D\cap(D+\alpha_{j})$ for $j = 1, 2, \ldots, k$. The set $C(\alpha)$ is then the Moran set associated with the collection $\cup_{k=1}^{\infty} \mathcal{F}_{k}$. Additionally, since the entries of the vector $R_{k}$ are the constant $|b|^{-1}$ for every $k$, the infimum across all entries of vectors in the sequence $(R_{k})_{k=1}^{\infty}$ is greater than zero. 

By Theorem~\ref{thm:moranDim}, we have $\dim_{H}C(\alpha) = \liminf_{k\rightarrow\infty} s_{k}$ where $s_{k}$ is the solution to $\sum_{\sigma\in W_{k}}c_{\sigma}^{s_{k}} = 1$. This is $(\Pi_{j=1}^{k}|D\cap(D+\alpha_{j})|)|b|^{-ks_{k}} = 1$. Solving this equation for $s_{k}$ reveals that $s_{k} = M_{k}(\alpha)$. Therefore $\dim_{H}C(\alpha) = \underline{\dim_{B}}C(\alpha)$. Similarly, $\dim_{P}C(\alpha) = \overline{\dim_{B}}C(\alpha)$. Recall that $s := \dim_{B}C(\alpha) = \dim_{H}C(\alpha) = \dim_{P}C(\alpha)$ since $C_{n, D}$ is self-similar (corollary 3.3 of \cite{F97}). We have that for all $\lambda \in [0, s]$, $\Phi_{n, D}^{-1}(\lambda)$ is a subset of both $\Psi_{n, D}^{-1}(\lambda)$ and $\Theta_{n, D}^{-1}(\lambda)$. The conclusion now holds by Corollary~\ref{cor:boxDense}. 
\end{proof}

\section{\textbf{Applications to Self-Similarity}} \label{section:selfSimSEP} 

In this section we prove Theorem~\ref{thm:sepSp} and Theorem~\ref{thm:sepGen}. Let us recall the notion of strong eventual periodicity for sequences of sets from the introduction. 

\begin{definition} 
A sequence $(a_{j})_{j\geq1}$ of integers is \textit{strongly eventually periodic} (SEP) if there exists a finite sequence $(b_{\ell})_{\ell = 1}^{p}$ and a nonnegative sequence $(c_{\ell})_{\ell = 1}^{p}$, where $p$ is a positive integer, such that
\begin{equation}
(a_{j})_{j\geq1} = (b_{\ell})\overline{(b_{\ell} + c_{\ell})_{\ell = 1}^{p}}, 
\end{equation}
where $\overline{(d_{\ell})_{\ell = 1}^{p}}$ denotes the infinite repetition of the finite sequence $(d_{\ell})_{\ell = 1}^{p}$.   
\end{definition}

The following is a convenient sufficient condition for the SEP property. 

\begin{lemma}\label{lem:monoSEP}
Let $(a_{j})_{j = 1}^{\infty}$ be a bounded sequence of nonnegative integers. If there exists a positive integer $q$ such that $a_{j} \leq a_{j + q}$ for all $j\geq1$, then $(a_{j})_{j\geq1}$ is SEP. 
\end{lemma}

\begin{proof}
Since $a_{j}$ is bounded for all $j$, there exists a positive integer $q$ such that $a_{j + kp}$ is constant for all $k\geq q$ and $j = 1, 2, \ldots, q$. It follows that there exists a positive integer $m$ such that $(a_{j})_{j=1}^{\infty} = a_{1}a_{2}\ldots a_{mq}\overline{a_{1+mq}a_{2+mq}\ldots a_{q + mq}}$. 

For each $j = 1, 2, \ldots, q$, there exist non-negative integers $v_{j + (m-\ell)q}$ where $\ell = 1, 2, \ldots m$ such that
\begin{align}
a_{j + mq} &= a_{j+ (m-1)q} +v_{j + (m-1)q} \\
&= a_{j + (m-2)q} + v_{j + (m-1)q} + v_{j+(m-2)q} \\
&\;\;\vdots \\
&= a_{j} + \sum_{\ell = 1}^{m} v_{j + (m-\ell)q}. 
\end{align}
By choosing $u_{j + qk} = \sum_{\ell = 1}^{m-k}v_{j + (m-\ell)q}$ for $j = 1, 2, \ldots, q$ and $k = 0, 1, \ldots, m - 1$, we obtain $(a_{j})_{j\geq1} = (a_{1}a_{2}\ldots a_{m})\overline{(a_{1} + u_{1})(a_{2}+u_{2})\ldots (a_{mq}+u_{mq})}$. 
\end{proof}

The following fact about self-similar sets is also useful.

\begin{lemma}\label{lem:selfSimShift} 
Let $a$ be an element of $\mathbb{R}^{m}$. If $S\subset\mathbb{R}^{m}$ is self-similar and is the attractor of the IFS given by $\{f_{i}(x) = r_{i}x + u_{i}\}_{i=1}^{N}$ where $r_{i}$ is a linear transformation from $\mathbb{R}^{m}$ to $\mathbb{R}^{m}$ and $u_{i}\in\mathbb{R}^{m}$, then the translation $S + a := \{x + a : x\in S\}$ is also a self-similar set and is the attractor of the IFS $\{g_{i}(x) = f_{i}(x) + (1 - r_{i})a\}_{i=1}^{n}$. 
\end{lemma}

\begin{proof}
Since $S$ is a nonempty compact set, so is $S + a$. Since attractors are unique we need only observe that $\cup_{i=1}^{N}g_{i}(S + a) = \cup_{i=1}^{N}(f_{i}(S) + a) = (\cup_{i=1}^{N}f_{i}(S)) + a = S + a$. 
\end{proof}

Fix an integer $n\geq2$ and suppose $D \subset \{0, 1, \ldots, n^{2}\}$. Given $\alpha = \pi(\alpha_{j})_{j=1}^{\infty}$ with $\alpha_{j}\in D-D$, we wish to relate the self-similarity of $C(\alpha) := C_{n, D} \cap (C_{n, D} + \alpha)$ to the condition that $(m - |\alpha_{j}|)_{j=1}^{\infty}$ is SEP. The first step is to express $C(\alpha)$ in terms of the sequence of sets $(D\cap(D + \alpha_{j}))_{j=1}^{\infty}$. Recall that $E_{n, D}$ denotes the extended restricted digit set generated by $(n, D)$ (see Definition~\ref{def:ExtDigitSet}). 

\begin{lemma} \label{lem:seqExpress} 
Fix an integer $n\geq2$ and suppose $D \subset \{0, 1, \ldots, n^{2}\}$. If $\alpha\in E_{n, D}$ has a unique radix expansion in base $(b, D - D)$, then $C_{n, D} \cap (C_{n, D} + \alpha) = \pi(\Pi_{j=1}^{\infty}D\cap(D + \alpha_{j}))$. 
\end{lemma}

\begin{proof}
Assume that $\alpha = \pi(\alpha_{j})_{j=1}^{\infty}$ where $\alpha_{j} \in \Delta:=D-D$ for all $j$ and $(\alpha_{j})_{j=1}^{\infty}$ is unique. Let $C(\alpha)$ denote $C_{n, D} \cap (C_{n, D} + \alpha)$. If $z\in C(\alpha)$, then $z \in C_{n, D}$ and $z \in C_{n, D} + \alpha$. On one hand, $z = \pi(z_{j})_{j=1}^{\infty}$ with $z_{j}\in D$ for all $j$. On the other, there exists $y \in C_{n, D}$ such that $z = y + \alpha$. Similarly, $y = \pi(y_{j})_{j=1}^{\infty}$ where $y_{j}\in D$ for all $j$. We conclude that $\alpha = \pi(z_{j} - y_{j})_{j=1}^{\infty}$ where $z_{j} - y_{j} \in \Delta$ for all $j$. Since the radix expansion in base $(b, \Delta)$ of $\alpha$ is unique, we conclude that $z_{j} - y_{j} = \alpha_{j}$ and, in particular, $z_{j} = y_{j} + \alpha_{j}$ for all $j$. This means that $z_{j}\in D\cap(D+\alpha_{j})$ for all $j$. We leave the inclusion $\pi(D\cap(D + \alpha_{j})) \subset C(\alpha)$ for the reader.
\end{proof}

Lemma~\ref{lem:seqExpress} will be used in the proofs of Theorem~\ref{thm:sepSp} and Theorem~\ref{thm:sepGen}. In the case of Theorem~\ref{thm:sepSp}, we transform the set $C_{n, D} \cap (C_{n, D} + \alpha)$ once more. 

\begin{definition}
Fix an integer $n\geq2$, $D = \{0, m\}$, $2\leq m \leq n^{2}$, and choose $\alpha \in E_{n, D}$ such that $\alpha$ has a unique radix expansion $\pi(\alpha_{j})_{j=1}^{\infty}$ in base $(b, D - D)$. We call $\gamma := \pi(\gamma_{j})_{j=1}^{\infty}$, where $\gamma_{j} := \min(D \cap (D + \alpha_{j}))$ for each $j$, the \textit{minimal element of $C_{n, D} \cap (C_{n, D} + \alpha)$}. 
\end{definition}

\begin{lemma}\label{lem:shiftForm} 
Fix an integer $n\geq2$ and suppose $D = \{0, m\}$ where $2\leq m \leq n^{2}$. If $\alpha \in E_{n, D}$ has a unique radix expansion $\pi(\alpha_{j})_{j=1}^{\infty}$ in base $(b, D - D)$ and $\gamma$ is the minimal element of $C(\alpha) := C_{n, D} \cap (C_{n, D} + \alpha)$. Then $C(\alpha) - \gamma = \{\pi(z_{j})_{j=1}^{\infty} : z_{j}\in D, z_{j} \leq m - |\alpha_{j}|\}$ and $C(\alpha) - \gamma$ is a subset of $C_{n, D}$. 
\end{lemma}

\begin{proof}
By Lemma~\ref{lem:seqExpress}, we can equate $C(\alpha)$ and $\pi(\Pi_{j=1}^{\infty}D\cap(D + \alpha_{j}))$. It follows that $C(\alpha) - \gamma$ is equal to $\pi(\Pi_{j=1}^{\infty}(D\cap(D + \alpha_{j})) - \gamma_{j})$. 

It is sufficient to show that for any $j\geq1$, the condition $z_{j} \leq m - |\alpha_{j}|$ and $z_{j}\in D$, holds if and only if $z_{j}$ is an element of $(D \cap (D + \alpha_{j})) - \gamma_{j}$. There are three possible values for $\alpha_{j}$ for each $j$. If $\alpha_{j} = m$, $(D \cap (D + \alpha_{j})) - \gamma_{j} = \{0\}$ and $m - |\alpha_{j}| = 0$ and it is now clear that the bi-implication holds. If $\alpha_{j} = -m$, we obtain the same equations. When $\alpha_{j} = 0$, $(D \cap (D + \alpha_{j})) - \gamma_{j} = \{0, m\}$ and $m - |\alpha_{j}| = m$. Since $\{0, m\} \cap \{z_{j} : z_{j} \leq m\}$ is equal to $\{0, m\}$, the bi-implication holds for this case too. 

Lastly, from these calculations we see that $(D\cap(D+\alpha_{j})) - \gamma_{j}$ is a subset of $\{0, m\} = D$ for all $j$, no matter the value . 
\end{proof}

We state an extended version of Theorem~\ref{thm:sepSp}. 

\begin{theorem}\label{thm:selfSimSEPSp} 
Fix an integer $n\geq 2$ and let $b:= -n+i$. Suppose $D = \{0, m\}$ where $2 \leq m \leq n^{2}$ and that $\alpha\in E_{n, D}$ is chosen such that $\alpha$ has a unique radix expansion in base $(b, \{0, \pm m\})$. Let $\gamma$ be the minimal element of $C(\alpha) := C_{n, D} \cap (C_{n, D} + \alpha)$. 

If $C(\alpha)$ is self-similar and is the attractor of an IFS containing the similarity $f(x) = rx + (1 - r)\gamma$, then $(m -|\alpha_{j}|)_{j=1}^{\infty}$ is SEP. 

Conversely, if $(m -|\alpha_{j}|)_{j=1}^{\infty}$ is SEP and so by definition can be written as $(a_{\ell})_{\ell=1}^{p}\overline{(a_{\ell} + u_{\ell})_{\ell=1}^{p}}$, then $C(\alpha)$ is self-similar and is the attractor of the IFS containing all maps of the form 
\begin{equation*}
f(x) = b^{-p}(x + \sum_{\ell=1}^{p}(y_{\ell}b^{p-\ell} + z_{\ell}b^{-\ell})-\gamma)+\gamma
\end{equation*}
where for each $\ell$, $y_{\ell}, z_{\ell}\in\{0, m\}$ such that $y_{\ell} \leq a_{\ell}$ and $z_{\ell} \leq u_{\ell}$.
\end{theorem}

We follow the proof strategy for theorem 1.2 in \cite{LYZ11}.

\begin{proof} 
Assume that $(m-|\alpha_{j}|)_{j=1}^{\infty}$ is an SEP sequence of integers. By definition there exist integers $a_{1}, a_{2}, \ldots, a_{p}$ and $u_{1}, u_{2}, \ldots, u_{p}$ such that $(m-|\alpha_{j}|)_{j\geq1} = a_{1}\ldots a_{p}\overline{(a_{\ell}+u_{\ell})_{\ell = 1}^{p}}$. 

By Lemma~\ref{lem:shiftForm}, $C(\alpha, \gamma) := C(\alpha) - \gamma$ is equal to $\{\pi(z_{j})_{j=1}^{\infty} : z_{j}\in D, z_{j} \leq m - |\alpha_{j}|\}$. 

Suppose $z$ is an element of $C(\alpha, \gamma)$. Therefore $z = \sum_{j=1}^{\infty}z_{j}b^{-j}$ with $z_{j}\leq (m-|\alpha_{j}|)$, $z_{j}\in\{0, m\}$ for each $j$. Let us organize the expansion as
\begin{equation}
z = \sum_{ell=1}^{p} x_{\ell}b^{-\ell} + \sum_{k=1}^{\infty}\sum_{\ell=1}^{p}x_{kp+\ell}b^{-(kp + \ell)}
\end{equation}
where $x_{\ell} \leq a_{\ell}$ and $x_{kp+\ell}\leq a_{\ell} + u_{\ell}$ for $\ell = 1, 2, \ldots, p$ and integers $k\geq1$. It follows that $x_{kp+\ell}$ can be decomposed into a sum $y_{kp+\ell} + z_{kp + \ell}$ where $y_{kp+\ell}, z_{kp + \ell}\in\{0, m\}$. Let us relabel $x_{\ell}$ as $y_{\ell}$ for $\ell = 1, 2, \ldots, p$. After making this substitution we can rearrange the terms to obtain 
\begin{equation}\label{eq:selfSimForm}
z =\sum_{k=1}^{\infty}\sum_{\ell=1}^{p}(y_{(k-1)p +\ell}b^{p-\ell} + z_{kp+\ell}b^{-\ell})b^{-kp}. 
\end{equation}
We now see that $z$ is an element of the attractor of the IFS $\{h(x) = b^{-p}(x + \sum_{\ell = 1}^{p}(y_{\ell}b^{p-\ell}+z_{\ell}b^{-\ell})) : y_{\ell}\leq a_{\ell}, z_{\ell}\leq u_{\ell}$, $y_{\ell}, z_{\ell}\in\{0, m\}$. We leave the other inclusion to the reader. By Lemma~\ref{lem:selfSimShift}, $C(\alpha)$ is self-similar because $C(\alpha, \beta)$ is self-similar.

Now let us assume $C(\alpha)$ is self-similar and is generated by an IFS containing the map $f_{1}(x) = r_{1}x + (1-r_{1})\gamma$. It follows from Lemma~\ref{lem:selfSimShift} that $C(\alpha, \gamma)$ is self similar and is generated by an IFS containing the map $g_{1}(x) = r_{1}x$. Observe that $\overline{0}$ is trivially SEP.
Suppose at least one of the entries of $(m-|\alpha_{j}|)_{j\geq1}$ is non-zero. 

For the sake of simple notation, let $C(\alpha, \gamma) := C(\alpha) - \gamma$. It follows from Lemma~\ref{lem:shiftForm} that $C(\alpha, \gamma)$ contains $mb^{-q}$ for some positive integer $q$. Therefore $mb^{-q}r_{1} = g_{1}(mb^{-q})$ is an element of $C(\alpha, \gamma)$ and, in particular, can be expressed as $\pi(mr_{1, j})_{j=1}^{\infty}$ where $r_{1, j}\in\{0, 1\}$. Isolating for $r_{1}$ yields
\begin{equation}
r_{1} = \sum_{j=1}^{\infty}r_{1, j}b^{-j+q}. 
\end{equation}

For every $s\geq1$, we have that
\begin{align}
\tau &:= g_{1}((m-|\alpha_{s}|)b^{-s}) \\
 &= (m-|\alpha_{s}|)r_{1, 1}b^{q-1-s} + \cdots+(m-|\alpha_{s}|)r_{1, q-s} + \sum_{j=q - s +1}^{\infty}(m-|\alpha_{s}|)r_{1, j}b^{-(j+s)} \label{eq:gImage}
\end{align}
is an element of $C(\alpha, \gamma)$. By Lemma~\ref{lem:shiftForm}, it is also an element of $C$. Therefore $\tau = \pi(d_{j})_{j=1}^{\infty}$ where $d_{j}\in\{0, m\}$. We now argue that the expansion in (\ref{eq:gImage}) is $\pi(d_{j})_{j=1}^{\infty}$. If $s > q - 1$, the desired result immediately follows from Lemma~\ref{lem:sepOne}. Suppose $s \leq q - 1$.  The number $\tau/b^{q-1-s}$ is equal to $\sum_{j=1}^{\infty}d_{j}b^{-j-(q-1-s)}$ and is also an element of $C$. It follows from Lemma~\ref{lem:sepOne} that $r_{1, 1} = r_{1, 2} = \cdots = r_{1, q-s} = 0$ and $(m-|\alpha_{s}|)r_{1, q-s+j} = d_{j}$ for all $j\geq1$. We obtain the desired result by multiplying back by $b^{q-1-s}$. 

Observe that $r_{1} \neq 0$ implies that $r_{1, t} = 1$ for some $t\geq1$. In particular, since $|r_{1}| < 1$, it must be that we can choose $t > q$. Otherwise, $r_{1}$ is a nonzero Gaussian integer and thus has magnitude greater than one. Fix such a $t$. It follows from the discussion above that $m - |\alpha_{s}| = d_{s-q+t}$ for all $s\geq 1$. Since $\tau$ is an element of $C(\alpha, \gamma)$, it follows from Lemma~\ref{lem:shiftForm} that $d_{s-q+t} \leq m - |\alpha_{s + (t - q)}|$ for all $s\geq1$. Therefore $m - |\alpha_{s}|  \leq m - |\alpha_{s + p}|$ for all $s\geq 1$ where $p:= t - q$ is a positive integer. We conclude that the sequence $(m - |\alpha_{j}|)_{j = 1}^{\infty}$ is SEP by Lemma~\ref{lem:monoSEP}. 
\end{proof}

Recall that for a fixed integer $n\geq2$ and a set $D \subset \{0, 1, \ldots, n^{2}\}$, the set $\Delta := D - D$ is called sparse if for all $\delta \neq \delta^{'}\in\Delta$ either $|\delta - \delta^{'}| > 2$ when $n\geq5$ or $|\delta-\delta^{'}| > 3$ when $n = 2, 3$ or $4$. In the statement of Theorem~\ref{thm:selfSimSEPSp}, if $m\geq4$, then $\Delta := D - D$ would be sparse and it would follow from Lemma~\ref{lem:extsep} that a radix expansion in base $(b, \Delta)$ of any $\alpha \in E_{n, D}$ is unique. The condition that $\Delta$ be sparse is stronger than the assumption that $\alpha$ is chosen to have a unique radix expansion. This stronger assumption does let us treat sets of digits $D$ that contain more than two elements. To state that the theorem, we recall the strong eventual periodicity of sequences of sets. 

\begin{definition} 
A sequence $(A_{j})_{j=1}^{\infty}$ of nonempty subsets of the integers is called \textit{strongly eventually periodic} (SEP) if there exist two finite sequences of sets $(B_{\ell})_{\ell = 1}^{p}$ and $(C_{\ell})_{\ell = 1}^{p}$, where $p$ is a positive integer, such that 
\begin{equation}
(A_{j})_{j=1}^{\infty} = (B_{\ell})_{\ell=1}^{p}\overline{(B_{\ell} + C_{\ell})_{\ell = 1}^{p}}, 
\end{equation}
where $B + C = \{b + c : b\in B, c\in C\}$ and $\overline{(D_{\ell})_{\ell = 1}^{p}}$ denotes the infinite repetition of the finite sequence of sets $(D_{\ell})_{\ell = 1}^{p}$. 
\end{definition}

\begin{remark} 
There is a connection between the SEP property for sequences of sets and the SEP property of sequences of integers. If the sequence $(m - |\alpha_{j}|)_{j=1}^{\infty}$ is SEP, then so is $(1 - |\alpha_{j}|/m)_{j=1}^{\infty}$. For each $j$, $1-|\alpha_{j}|/m = |D\cap(D+\alpha_{j})| - 1$. It can be shown that if $(|D_{j}| - 1)_{j=1}^{\infty}$ is SEP and that each $D_{j}$ is an arithmetic progression with a common step size for all $j$, then $(D_{j} - \gamma_{j})_{j=1}^{\infty}$ is SEP where $\gamma_{j} := \min D_{j}$. In the context of Theorem~\ref{thm:selfSimSEPSp}, it follows that $(D\cap(D+\alpha_{j}))_{j=1}^{\infty}$ is an SEP sequence of sets.
\end{remark}

The following is a convenient sufficient condition for the SEP property that is analogous to lemma~\ref{lem:monoSEP}.  

\begin{lemma}\label{lem:sepPP} 
Let $(A_{j})_{j=1}^{\infty}$ be a sequence of subsets of $\mathbb{Z}$ such that there exists $q$ such that $A_{j} \subset A_{j + q}$ and a bound, for all $j$, on the cardinality of $A_{j}$. If there exists a sequence of subsets $(U_{j})_{j\geq1}$ of $\mathbb{Z}$ such that $A_{j} + U_{j} = A_{j+q}$, then $(A_{j})_{j=1}^{\infty}$ is SEP. 
\end{lemma}

\begin{proof}
The uniform bound on the cardinality of $A_{j}$ for all $j$ together with the assumption $A_{j} \subset A_{j+q}$ implies that the existence of a postive integer $p$, such that $A_{j + q\ell} = A_{j + q(\ell + k)}$ for all $k\geq0$ and $j = 1, 2, \ldots, q$. It follows that for $j = 1, 2, \ldots, q$, 
\begin{align}
A_{j + qp} &= A_{j+ q(p-1)} +U_{j + q(p-1)} \\
&= A_{j + q(p-2)} + U_{j + q(p-1)} + U_{j+q(p-2)} \\
&\;\;\vdots \\
&= A_{j} + \sum_{\ell = 1}^{p} U_{j + q(p-\ell)}. 
\end{align}
By choosing $V_{j + qk} = \sum_{\ell = 1}^{p-k}U_{j + q(p-\ell)}$ for $j = 1, 2, \ldots, q$ and $k = 0, 1, \ldots, p - 1$, we obtain $(A_{j})_{j\geq1} = (A_{1}A_{2}\ldots A_{qp})\overline{(A_{1} + V_{1})(A_{2}+V_{2})\ldots (A_{qp}+V_{qp})}$. 
\end{proof}

\begin{theorem}\label{thm:selfSimSEPHom} 
Fix an integer $n\geq2$ and suppose $D\subset\{0, 1, \ldots, n^{2}\}$ is such that $\Delta := D - D$ is sparse. Let $\alpha = \pi(\alpha_{j})_{j=1}^{\infty}$ where $\alpha_{j}\in\Delta$. The set $C(\alpha) := C_{n, D} \cap (C_{n, D} + \alpha)$ is self similar and is the attractor of an IFS of the form $\{f_{i}(x) = b^{-p}x + u_{i}\}$ where $p$ is a positive integer and each $u_{i}$ is a complex number if and only if the sequence of sets $((D\cap (D + \alpha_{j})) - \beta_{j})_{j=1}^{\infty}$ is SEP for some $\beta = \pi(\beta_{j})_{j=1}^{\infty} \in C(\alpha)$. 

Moreover, since $((D\cap (D + \alpha_{j})) - \beta_{j})_{j=1}^{\infty}$ is SEP and by definition can be written as $(A_{\ell})_{\ell=1}^{p}\overline{(A_{\ell} + U_{\ell})_{\ell=1}^{p}}$, the set $C(\alpha)$ is the attractor of the IFS containing all maps of the form 
\begin{equation*}
f(x) = b^{-p}(x + \sum_{\ell=1}^{p}(a_{\ell}b^{p-\ell} + u_{\ell}b^{-\ell})-\beta)+\beta
\end{equation*}
where $a_{\ell} \in A_{\ell}$ and $u_{\ell} \in U_{\ell}$ for each $\ell$. 
\end{theorem}

We follow the proof strategy for theorem 1.2 in \cite{PP14}. 

\begin{proof} 
Suppose $\beta = \sum_{j=1}^{\infty}\beta_{j}b^{-j}\in C(\alpha)$ is such that $(D_{j})_{j=1}^{\infty} := ((D\cap(D+\alpha_{j})) - \beta_{j})_{j=1}^{\infty}$ is an SEP sequence of sets. By definition there exist sets $A_{1}, A_{2}, \ldots, A_{p}$ and $U_{1}, U_{2}, \ldots, U_{p}$ such that $(D_{j})_{j=1}^{\infty} = A_{1}\ldots A_{p}\overline{(A_{\ell}+U_{\ell})_{\ell = 1}^{p}}$. 

Since $\Delta$ is sparse it follows from Lemma~\ref{lem:extsep} that $\alpha$ has a unique radix expansion in base $(b, \Delta)$. Therefore, by Lemma~\ref{lem:seqExpress}, we have $C(\alpha) = \pi(\Pi_{j=1}^{\infty}D\cap(D + \alpha_{j}))$. It follows that $C(\alpha, \beta) := C(\alpha) - \beta$ is equal to  $\pi(\Pi_{j=1}^{\infty}(D\cap(D + \alpha_{j}))-\beta_{j})$. 

Suppose $z$ is an element of $C(\alpha, \beta)$. Therefore $z = \sum_{j=1}^{\infty}z_{j}b^{-j}$ with $z_{j}\in A_{j}$. By assumption, we can write
\begin{equation}
z = \sum_{j=1}^{p} a_{0, j}b^{-j} + \sum_{k=1}^{\infty}\sum_{j=1}^{p}(a_{k, j} + u_{k, j})b^{-(kp + j)}
\end{equation}
where $a_{k, j}\in A_{j}$ and $u_{k, j}\in U_{j}$. We can rearrange the terms to obtain 
\begin{equation}\label{eq:selfSimForm}
z =\sum_{k=0}^{\infty}\sum_{j=1}^{p}(a_{k, j}b^{p-j} + u_{k+1, j}b^{-j})b^{-(k+1)p}. 
\end{equation}
We now see that $z$ is an element of the attractor of the IFS $\{h(x) = b^{-p}(x + \sum_{\ell = 1}^{p}(a_{\ell}b^{p-\ell}+u_{\ell}b^{-\ell})) : a_{\ell}\in A_{\ell}, u_{\ell}\in U_{\ell}\}$. We leave the other inclusion to the reader. By Lemma~\ref{lem:selfSimShift}, $C(\alpha)$ is self-similar because $C(\alpha, \beta)$ is self-similar. 

Conversely, assume that $C(\alpha)$ is generated by the IFS of the form $\{f_{i}(x) = b^{-p}x + v_{i}\}_{i=1}^{N}$. Let $\beta = \frac{v_{1}}{1-b^{-p}}$. By Lemma~\ref{lem:selfSimShift}, $C(\alpha, \beta)$ is generated by the IFS $\{g_{i}(x) = b^{-p}x+u_{i}\}_{i=1}^{N}$ where $u_{i} = v_{i} - v_{1}$. Therefore $g_{1}(x) = b^{-p}x$. It follows that $C(\alpha, \beta)$ contains the origin. Therefore $\beta$ is an element of $C(\alpha)$ and has the expansion $\sum_{j=1}\beta_{j}b^{-j}$ where $\beta_{j}\in D\cap(D+\alpha_{j})$. The presence of the origin also means $g_{i}(0) = u_{i}$ is an element of $C(\alpha, \beta)$ for $i = 1, 2, \ldots, N$. It follows that each $u_{i}$ can be expanded into $\sum_{j=1}^{\infty}u_{i, j}b^{-j}$ where $u_{i, j}\in D_{j} := (D\cap(D+\alpha_{j}))-\beta_{j}$. 

Define $U_{j} := \{u_{i, j+p}: i = 1, 2, \ldots, N\}$. We now argue that $D_{j} + U_{j} = D_{j + p}$ for all $j\geq1$. 

Let $d\in D_{k} + U_{k}$. There exists $z_{k}\in D_{k}$ and $u_{i, k+p}\in U_{k}$, for some $i$, such that $d = z_{k} + u_{i, k+p}$. Choose $z\in C(\alpha, \beta)$ such that its $k$th digit is $z_{k}$. Then $g_{i}(z) = \sum_{j=1}^{p}u_{i, j}b^{-j} + \sum_{j=1}^{\infty}(z_{j} + u_{i, j+p})b^{-(p+j)}$ is an element of $C(\alpha, \beta)$. Therefore it also has an expansion $\sum_{j=1}^{\infty}d_{j}b^{-j}$ where $d_{j}\in D_{j}$ for each $j$. It follows that
\begin{equation}
\sum_{j=1}^{\infty}z_{j}b^{-(p+j)} = \sum_{j=1}^{\infty}(d_{j} - u_{i, j})b^{-j}. 
\end{equation}
We deduce that $z_{k} = d_{k+p} - u_{i, k + p}$ by applying Lemma~\ref{lem:extsep}. Therefore $d = d_{k+p}\in D_{k+p}$. Observe now that if $D_{k} + U_{k} \subsetneq D_{k+p}$, then 
\begin{align}
\cup_{i=1}^{N}g_{i}(C(\alpha, \beta)) &\subset \pi(D_{1}\times\cdots\times D_{k+p-1} \times (D_{k} + U_{k}) \times \Pi_{j=1}^{\infty}D_{k + p + j}) \\
&\subsetneq \pi(\Pi_{j=1}^{\infty}D_{j}) = C(\alpha, \beta). 
\end{align}
This is a contradiction. Therefore $D_{j} + U_{j} = D_{j+p}$ for all $j\geq1$. 

Since $g_{1}$ is among the contractions generating $C(\alpha, \beta)$, it follows that $D_{j} \subset D_{j+p}$ for all $j\geq1$. We also recall that $D_{j} \subset \{-n^{2}, -n^{2}+1, \ldots, n^{2}-1, n^{2}\}$ for all $j\geq1$. By Lemma~\ref{lem:sepPP}, it follows that $(D_{j})_{j=1}^{\infty}$ is SEP. 
\end{proof}

\begin{remark}
The assumption that $b^{-p}$ is the linear factor of one of the contractions, while sufficient, may be unnecessary. 
\end{remark}

Let us now recall the significance of self-similarity by recalling Theorem~\ref{thm:strongSimDim}. A finite collection of similarities $\{f_{i}\}$ satisfies the strong separation condition (see Definition \ref{def:strongSep}), then the attractor of the IFS has box-counting dimension, Hausdorff dimension, and packing dimension equal to the unique $s$ that satisfies $\sum_{i=1}^{N}r_{i}^{s} = 1$ where each $r_{i}$ is the contraction factor of the similarity $f_{i}$. We can use this fact from fractal geometry to find the fractal dimension of $C_{n, D} \cap (C_{n, D} + \alpha)$ for a suitable pair $n\geq2$ and $D\subset\{0, 1, \ldots, n^{2}\}$, and choice of $\alpha$. We present two results, one aligned with Theorem~\ref{thm:selfSimSEPSp} and another with Theorem~\ref{thm:selfSimSEPHom}. 

\begin{corollary}~\label{cor:dimSp}
Fix an integer $n\geq 2$ and let $b:= -n+i$. Suppose $D = \{0, m\}$ where $2 \leq m \leq n^{2}$ and that $\alpha\in E_{n, D}$ is chosen such that $\alpha$ has a unique radix expansion in base $(b, \{0, \pm m\})$. Let the finite collections of integers $\{a_{\ell}\}_{\ell=1}^{p}$ and $\{u_{\ell}\}_{\ell=1}^{p}$ be such that $(m - |\alpha_{j}|)_{j=1}^{\infty} = (a_{\ell})_{\ell=1}^{p}\overline{(a_{\ell} + u_{\ell})_{\ell=1}^{p}}$. Then the following equation holds. 
\begin{equation}
\dim_{B}C(\alpha) = \dim_{H}C(\alpha) = \dim_{P}C(\alpha) =  \frac{\log{2}\sum_{\ell=1}^{p}(a_{\ell}+u_{\ell})}{mp\log{|b|}}
\end{equation}
where $C(\alpha) := C_{n, D} \cap (C_{n, D} + \alpha)$. 
\end{corollary}

\begin{proof} 
Let $\gamma$ denote the minimal element of $C(\alpha)$. By Theorem~\ref{thm:selfSimSEPSp}, $C(\alpha)$ is self-similar and in particular, $C(\alpha, \gamma)$ is the attractor of the collection of similarities of the form $f(x) = b^{-p}(x + \sum_{\ell=1}^{p}(y_{\ell}b^{p-\ell} + z_{\ell}b^{-\ell}))$ where $y, u_{\ell} \in \{0, m\}$ such that $y_{\ell} \leq a_{\ell}$ and $z_{\ell} \leq u_{\ell}$ for each $\ell$. 

Suppose $h$ and $g$ are functions of that form and suppose there exist $x_{1}$, $x_{2}\in C(\alpha, \gamma)$ such that $h(x_{1}) = g(x_{2})$. Let us denote parameters of $h$ and $g$ by $y_{\ell}^{(h)}, z_{\ell}^{(h)}$ and $y_{\ell}^{(g)}, z_{\ell}^{(g)}$ respectively. By Lemma~\ref{lem:shiftForm}, $x_{k} = \pi(x_{j}^{(k)})_{j=1}^{\infty}$ for $k = 1, 2$ with $x_{j}^{(k)}\in \{0, m\}$ such that $x_{j}^{(k)} \leq m - |\alpha_{j}|$ for each $j$.  It follows that 
\begin{align}
h(x_{1}) &= 0._{b}y_{1}^{(h)}\ldots y_{p}^{(h)}(z_{1}^{(h)} + x_{1}^{(1)})\ldots (z_{p}^{(h)} + x_{p}^{(1)})x_{p+1}^{(1)}x_{p+2}^{(1)}\ldots \label{eq:hy}\\
g(x_{2}) &= 0._{b}y_{1}^{(g)}\ldots y_{p}^{(g)}(z_{1}^{(g)} + x_{1}^{(2)})\ldots (z_{p}^{(g)} + x_{p}^{(2)})x_{p+1}^{(2)}x_{p+2}^{(2)}\ldots \label{eq:gy}
\end{align}
where we recall that for $k = 1, 2$, $x_{\ell}^{(k)} \leq a_{\ell}$ for $\ell = 1, 2, \ldots, p$ and $x_{\ell}^{(k)} \leq a_{\ell} + u_{\ell}$ for $\ell > p$. It follows that both (\ref{eq:hy}) and (\ref{eq:gy}) are radix expansions in base $(b, D)$. Since $m \geq 2$, it follows from Lemma~\ref{lem:sepOne} that $y_{\ell}^{(h)} = y_{\ell}^{(g)}$ and $z_{\ell}^{(h)} + x_{\ell}^{(1)} = z_{\ell}^{(g)} + x_{\ell}^{(2)}$ for each $\ell = 1, 2, \ldots, p$. Suppose for some $\ell = 1, 2, \ldots, p$, we have $z_{\ell}^{(h)} \neq z_{\ell}^{(g)}$. Without loss of generality, suppose $z_{\ell}^{(h)} = m$. It follows that $x_{\ell}^{(2)} = m$. Therefore $u_{\ell} = a_{\ell} = m$. This is contradiction because $a_{\ell} + u_{\ell}  = m - |\alpha_{\ell + p}|$, which can never be $2m$. We conclude that $z_{\ell}^{(h)} = z_{\ell}^{(g)}$. We conclude that $h = g$. If $h \neq g$, then $h(C(\alpha, \gamma))$ and $g(C(\alpha, \gamma))$ are disjoint. 

By Theorem~\ref{thm:strongSimDim}, the box-counting, Hausdorff, and packing dimensions of $C(\alpha, \beta)$, and thus $C(\alpha)$ are both equal to the value $s$ satisfying $\Pi_{\ell=1}^{p}2^{(a_{\ell}+u_{\ell})/m} = |b|^{sp}$.
\end{proof}

\begin{corollary}\label{cor:dimHom} 
Fix an integer $n\geq2$ and suppose $D\subset\{0, 1, \ldots, n^{2}\}$ is such that $\Delta := D - D$ is sparse. Let $\alpha = \pi(\alpha_{j})_{j=1}^{\infty}$ where $\alpha_{j}\in\Delta$. 

Let the finite collections of sets $\{A_{\ell}\}_{\ell=1}^{p}$ and $\{U_{\ell}\}_{\ell=1}^{p}$ be such that $((D\cap(D+\alpha_{j}) - \beta_{j})_{j=1}^{\infty} = (A_{\ell})_{\ell=1}^{p}\overline{(A_{\ell} + U_{\ell})_{\ell=1}^{p}}$ for some sequence of $\beta_{j} \in D \cap (D +\alpha_{j})$. If $|A_{\ell} + U_{\ell}| = |A_{\ell}||U_{\ell}|$ for each $\ell$, then 
\begin{equation}
\dim_{B}C(\alpha) = \dim_{H}C(\alpha) = \dim_{P}C(\alpha) = \frac{\sum_{\ell=1}^{p}(\log{|A_{\ell}||B_{\ell}|})}{p\log{|b|}}
\end{equation}
where $C(\alpha) := C_{n, D} \cap (C_{n, D} + \alpha)$. 
\end{corollary}

We know that the sets $A_{\ell}$ and $U_{\ell}$ exist by the definition of SEP. We remark that the condition $|A_{\ell} + U_{\ell}| = |A_{\ell}||U_{\ell}|$ is equivalent to the condition that every element of $A_{\ell} + U_{\ell}$ has a unique decomposition of the form $a_{\ell} + u_{\ell}$ where $a_{\ell}\in A_{\ell}$ and $u_{\ell}\in U_{\ell}$. 

\begin{proof}
By Theorem~\ref{thm:selfSimSEPHom}, $C(\alpha)$ is self-similar and in particular, $C(\alpha, \beta) := C(\alpha) - \beta$ is the attractor of the collection of similarities of the form 
$f(x) = b^{-p}(x + \sum_{\ell=1}^{p}(a_{\ell}b^{p-\ell} + u_{\ell}b^{-\ell}))$ where $a_{\ell}$ and $u_{\ell}$ can be any of the elements in $A_{\ell}$ and $U_{\ell}$ respectively. 

Suppose $h$ and $g$ are functions of that form and suppose there exist $x_{1}$, $x_{2}\in C(\alpha, \beta)$ such that $h(x_{1}) = g(x_{2})$.  Let us denote the parameters of $h$ and $g$ by $a_{\ell}^{(h)}, u_{\ell}^{(h)}$ and $a_{\ell}^{(g)}, u_{\ell}^{(g)}$ respectively. By  $x_{k} = \pi(x_{j}^{(k)})_{j=1}^{\infty}$ for $k=1, 2$ where $x_{j}^{(k)}\in (D\cap(D+\alpha_{j})) - \beta_{j}$. It follows that 
\begin{align}
h(x_{1}) &= 0._{b}a_{1}^{(h)}\ldots a_{p}^{(h)}(u_{1}^{(h)} + x_{1}^{(1)})\ldots (u_{p}^{(h)} + x_{p}^{(1)})x_{p+1}^{(1)}x_{p+2}^{(1)}\ldots \label{eq:ha} \\
g(x_{2}) &= 0._{b}a_{1}^{(g)}\ldots a_{p}^{(g)}(u_{1}^{(g)} + x_{1}^{(2)})\ldots (u_{p}^{(g)} + x_{p}^{(2)})x_{p+1}^{(2)}x_{p+2}^{(2)}\ldots \label{eq:ga}
\end{align}
where we recall that for $k=1, 2$, $x_{\ell}^{(k)}$ is an element of $A_{\ell}$ for $\ell = 1, 2, \ldots, p$ and of $A_{\ell} + U_{\ell}$ for $\ell > p$. It follows that both (\ref{eq:ha}) and (\ref{eq:ga}) are radix expansions in base $(b, \Delta)$. Since $\Delta$ is sparse, it follows from Lemma~\ref{lem:extsep} that $a_{\ell}^{(h)} = a_{\ell}^{(g)}$ and $u_{\ell}^{(h)} + x_{\ell}^{(1)} = u_{\ell}^{(g)} + x_{\ell}^{(2)}$ for each $\ell = 1, 2, \ldots, p$. The assumption $|A_{\ell} + U_{\ell}| = |A_{\ell}||U_{\ell}|$ implies that $x_{\ell}^{(1)} = x_{\ell}^{(2)}$ and, in particular, $u_{\ell}^{(h)} = u_{\ell}^{(g)}$. We conclude that $h = g$. If $h \neq g$, then $h(C(\alpha, \beta))$ and $g(C(\alpha, \beta))$ are disjoint. 

By Theorem~\ref{thm:strongSimDim}, the box-counting, Hausdorff, and packing dimensions of $C(\alpha, \beta)$, and thus $C(\alpha)$ are both equal to the value $s$ satisfying $\Pi_{\ell=1}^{p}|A_{\ell}||U_{\ell}| = |b|^{sp}$.
\end{proof}

We end by computing the dimension of the set $C_{n, D} \cap (C_{n, D} + \alpha)$ given in Example~\ref{ex:boxEx} a second time. 

\begin{example} 
Let $b=-3+i$. Suppose we choose $D = \{0, 4\}$. Therefore we have $\Delta = \{-4, 0, 4\}$. The intersection of $C$ and its translation by $\alpha = \frac{-28+24i}{-19+26i} = 0.\overline{-404}$ is nonempty. We compute its box-counting, Hausdorff, and packing dimensions. 

Not only is $m = 4 \geq 2$, but by Lemma~\ref{lem:extsep}, the radix expansion of $\alpha$ in base $(b, \Delta)$ is unique. Therefore we can apply Corollary~\ref{cor:dimSp}. The sequence $(4 - |\alpha_{j}|)_{j=1}^{\infty}$ is equal to $\overline{040}$. Matching the notation of Corollary~\ref{cor:dimSp}, we have $p = 3$, $a_{1} = a_{3} = 0$, $a_{2} = 4$ and $u_{1} = u_{2} = u_{3} = 0$. Therefore

\begin{equation}
\dim_{B}(C(\alpha)) = \frac{4\log{2}}{12\log{10}} = \frac{\log{2}}{3\log{10}}. 
\end{equation}
\end{example}

\section*{\textbf{Acknowledgement}}
The author would like to thank Professor Derong Kong at Chongqing University for his observations that improved the strength of some of the results and the quality of the manuscript. The author is also grateful for the continued support of his doctoral advisors, Professor Francesco Cellarosi and Professor James Mingo, whose comments improved the readability of the manuscript.

\nocite{*}
\renewcommand{\bibname}{References}

\bibliographystyle{plain} 
\clearpage
\bibliography{ReferencesV3}

\end{document}